\documentclass[a4paper,11pt,DIV=12]{scrartcl}
\usepackage[english]{babel}

\usepackage[utf8]{inputenc}
\usepackage[T1]{fontenc}
\usepackage{textcomp}

\usepackage{amsmath, amssymb, amsfonts, amsbsy, latexsym,amsthm}

\setkomafont{disposition}{\normalfont\bfseries}

\usepackage{xcolor}
\usepackage{graphicx}
\usepackage[labelformat=simple]{subfig}

\usepackage{tikz}
\usetikzlibrary{arrows}
\usetikzlibrary{positioning}

\tikzstyle{ticklabel}=[right=-1pt,font=\tiny]

\newcommand{\imageWithBorder}[2]{
  \begin{tikzpicture}
    \node[anchor=south west,inner sep=0] (image) at (0,0) {%
      \includegraphics[width=#1]{#2}
    };
    \begin{scope}[x={(image.south east)},y={(image.north west)}]
      \draw[ultra thin,black!33] (0,0) -- (1,0) -- (1,1) -- (0,1) -- cycle;
    \end{scope}
  \end{tikzpicture}
}

\definecolor{red0hsb}{hsb}{0,1,1}
\definecolor{red1hsb}{hsb}{1,1,1}
\definecolor{greenhsb}{hsb}{0.33,1,1}
\definecolor{bluehsb}{hsb}{0.66,1,1}

\newcommand{\imageWithRedBlueColorbarExp}[5]{
  \begin{tikzpicture}
    \node[anchor=south west,inner sep=0] (image) at (0,0) {%
      \includegraphics[width=#1]{#2}
    };
    \begin{scope}[x={(image.south east)},y={(image.north west)}]
      \draw[ultra thin,black!33, use as bounding box] (0,0) -- (1,0) -- (1,1) -- (0,1) -- cycle;

      \shade[top color=red, bottom color=white] (1.025,.5) rectangle (1.075,1);
      \shade[top color=white, bottom color=blue] (1.025,.0) rectangle (1.075,.5);

      \node[anchor=west,ticklabel] at (1.06,-0.03) {$\times 10^{#5}$};

      \draw[very thin, black] (1.025,0) -- (1.075,0) -- (1.075,1) -- (1.025,1) -- cycle;

      \foreach \p in #4 {
        \pgfmathparse{0.5 + 0.5 * \p};%
        \pgfmathsetmacro\position{\pgfmathresult}
        \pgfmathparse{\p * #3};%
        \pgfmathsetmacro\value{\pgfmathresult}
        \node at (1.075,\position) [ticklabel] {$\pgfmathprintnumber[fixed,precision=2]{\value}$};
        \draw[very thin, black] (1.0583,\position) -- (1.075,\position);
        \pgfmathparse{0.5 - 0.5 * \p};%
        \pgfmathsetmacro\position{\pgfmathresult}
        \node at (1.075,\position) [ticklabel] {$-\pgfmathprintnumber[fixed,precision=2]{\value}$};
        \draw[very thin, black] (1.0583,\position) -- (1.075,\position);
      };
      \node at (1.075,0.5) [ticklabel] {$0$};
      \draw[very thin, black] (1.0583,.5) -- (1.075,.5);
    \end{scope}
\end{tikzpicture}
}

\usepackage{url}
\usepackage{hyperref}
\hypersetup{
colorlinks,%
citecolor=black,%
filecolor=black,%
linkcolor=black,%
urlcolor=black
}

\usepackage[sort,numbers]{natbib}

\usepackage{enumerate}

\allowdisplaybreaks[1]

\newtheorem{theorem}{Theorem}[section]

\newtheorem{lemma}[theorem]{Lemma}

\newtheorem{remark}[theorem]{Remark}

\usepackage{listings}
\definecolor{lightblue}{rgb}{0,0,0.8}
\definecolor{darkgreen}{rgb}{0,0.5,0}
\definecolor{darkgrey}{rgb}{0.35,0.35,0.35}
\definecolor{lightgrey}{rgb}{0.925,0.925,0.925}
\usepackage{marginnote}

\usepackage{dirtree}

\providecommand{\RR}{\mathbb{R}}
\providecommand{\ZZ}{\mathbb{Z}}
\providecommand{\NN}{\mathbb{N}}
\providecommand{\CC}{\mathbb{C}}

\DeclareMathOperator{\supp}{supp}

\newcommand{\dif}[1]{\mathop{\mathrm{d}#1}}
\renewcommand{\i}{\operatorname{i}}
\newcommand{\e}{\operatorname{e}}

\newcommand{\cC}{\mathcal{C}}

\newcommand{\cF}{\mathcal{F}}
\newcommand{\cG}{\mathcal{G}}
\newcommand{\cH}{\mathcal{H}}
\newcommand{\cI}{\mathcal{I}}

\newcommand{\cSH}{\mathcal{SH}}

\newcommand{\tT}{{\text{\tiny$\operatorname{T}$}}}

\newcommand{\cone}{\cC}

\newcommand{\fftn}{\text{\texttt{fft2}}}

\newcommand{\ifftn}{\text{\texttt{ifft2}}}

\title{Fast Finite Shearlet Transform: a tutorial}
\author{
  Sören Häuser
  \thanks{
  Fachbereich für Mathematik, 
  Technische Universität Kaiserslautern, 
  Paul-Ehrlich-Str. 31,
  67663 Kaiserslautern, 
  Germany,
  \{haeuser,steidl\}@mathematik.uni-kl.de  
  }
  \and
  Gabriele Steidl
  \footnotemark[1]
}
\begin{document}

\maketitle

\tableofcontents

\begin{abstract}

\end{abstract}

\section{Introduction} \label{sec:introduction}
Directional multiscale representation of images to address curved singularities has received much attention in harmonic analysis in the last 25 years.
In particular, shearlets \cite{GKL06} and curvelets \cite{CDDY06} provide an optimally sparse approximation of carton-like images, that is
\begin{equation*}
  \lVert f - f_N\rVert_{L_2}^2 \leq C  N^{-2} (\log N)^3 
\quad 
  \text{as } N \to \infty,
\end{equation*}
where $f_N$ is the nonlinear shearlet approximation of a function $f$
from this class obtained by taking the $N$ largest shearlet coefficients in absolute value.
Shearlets possess a uniform construction for both the continuous and the discrete setting.
They further stand out since they stem
from a square-integrable group representation
\cite{DKMSST08} and have the corresponding useful mathematical properties.
Moreover, similarly as wavelets are related to Besov spaces via atomic decompositions,
shearlets correspond to certain function spaces, the so-called shearlet coorbit spaces
\cite{DKST09}.

Figure~\ref{fig:edgeDetection} illustrates the directional information contained in the shearlet coefficients. 
Shearlets have been applied to a wide field of image processing tasks,
e.g.,
denoising \cite{ELC09,DCGG11}, 
inversion of the Radon transform \cite{CEGL10,ECL09},
inverse halftoning \cite{EPH09},
deconvolution \cite{PEH09},
geometric separation \cite{DK09},
inpainting \cite{KKL13}
and many more.
A detailed summary can be found in \cite{EL12}.
In \cite{HS13} the authors show how the directional information encoded by the shearlet transform
can be used in image segmentation. 
To this end, we introduce a simple discrete shearlet transform which translates the shearlets over the full grid
at each scale and for each direction. Using the FFT this transform can be still realized in a fast way.

This tutorial explains the details behind the \textsc{Matlab}-implementation of the transform
and shows how to apply the transform. The software is available for free under the GPL-license at
\begin{center}
	\url{http://www.mathematik.uni-kl.de/imagepro/software/}
\end{center}
In analogy with other transforms we named the software \emph{FFST}---\textbf{F}ast \textbf{F}inite \textbf{S}hearlet \textbf{T}ransform.
The package provides a fast implementation of the finite (discrete) shearlet transform.

For shearlets there are currently three toolboxes available.
They are
\begin{description}
  \item[Local Shearlet Toolbox\footnotemark] 
    \footnotetext{\url{http://www.math.uh.edu/~dlabate/software.html}}
    developed by  Easley, Labate and Lim. This was the first shearlet implementation, for details see \cite{ELL08}.
  \item[ShearLab\footnotemark]\footnotetext{\url{http://www.shearlab.org}}
    developed by Donoho, Kutyniok, Lim, Shahram, Zhuang and Reisenhofer. This package consists of three different implementations: One is implemented on pseudo-polar grids, one on Cartesian grids and the newest one using compactly supported shearlets, for details see \cite{KSD09,Lim10,KSZ12,KLR14}.
  \item[Fast Finite Shearlet Transform (FFST)\footnotemark] 
    \footnotetext{\url{http://www.mathematik.uni-kl.de/imagepro/software/}}
    developed by the authors. The first fully finite and translation invariant shearlet implementation, described in \cite{HS13,Hae14} and in this tutorial.
\end{description}
\begin{figure}[ht]
  \centering
  \subfloat
  [Geometric shapes with different edge orientations.]
  {
    \label{fig:formen}
    \imageWithBorder{0.3\textwidth}{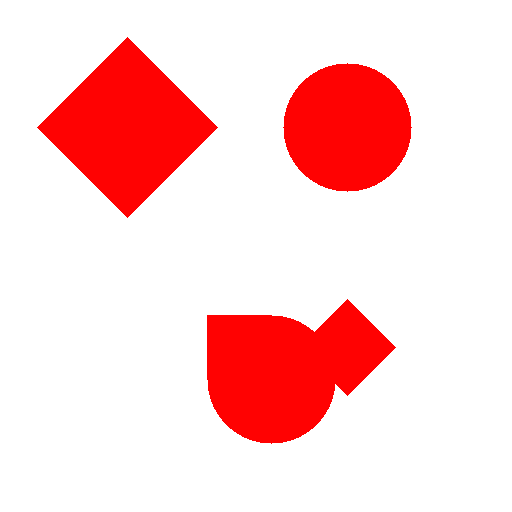}
  }
  \hspace{0.2cm}
  \subfloat
  [Sum of shearlet coefficients for $j=3$ ($a=\frac{1}{64}$) over all $k$ ($s$).]
  {
    \label{fig:formen_all_directions} 
    \imageWithBorder{0.3\textwidth}{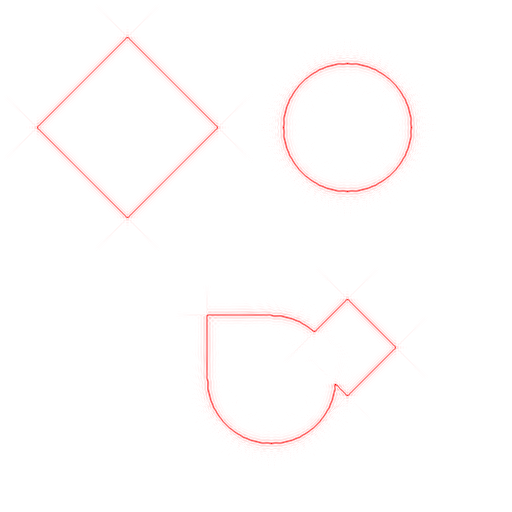}
  }
  \\
  \subfloat
  [Shearlet coefficients for $j=3$ ($a=\frac{1}{64}$) and $k=-8$ ($s = -1$).]
  {
    \label{fig:formen_one_correct_direction}
    \imageWithBorder{0.3\textwidth}{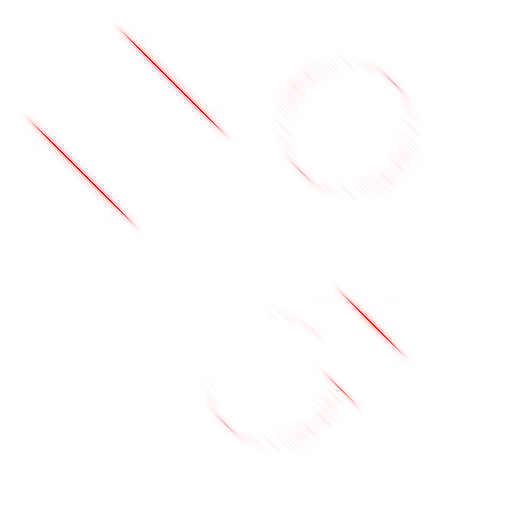}
  }
  \hspace{0.2cm}
  \subfloat
  [Shearlet coefficients for $j=3$ ($a=\frac{1}{64}$) and $k=6$ ($s = \tfrac{3}{4}$).]
  {
    \label{fig:formen_one_false_direction}
    \imageWithBorder{0.3\textwidth}{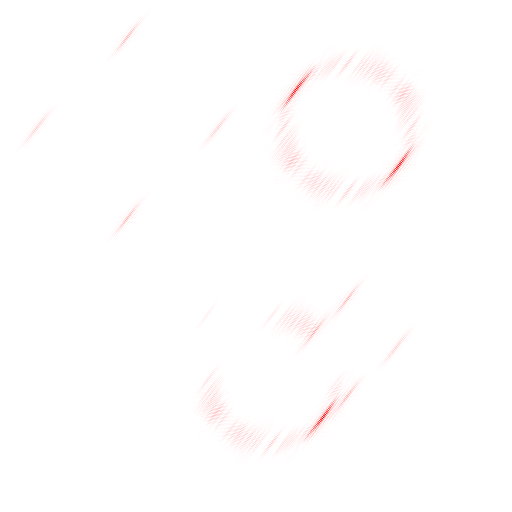}
  }
  \caption{Shearlet coefficients can detect edges with different orientations.}
  \label{fig:edgeDetection}
\end{figure}

Recall that the \emph{Fourier transform} $\mathcal{F}\colon L_2(\RR^2)\to L_2(\RR^2)$ and the inverse transform are defined by
\begin{alignat*}{4}
  \mathcal{F}f(\omega)
&=
  \hat{f}(\omega)
&:=&
  \int_{\RR^2} f(t)e^{-2\pi i\langle\omega,t\rangle} dt,
\\
  \mathcal{F}^{-1}\hat{f}(\omega)
&=
  f(t)
&=&
  \int_{\RR^2} \hat{f}(\omega)e^{2\pi i\langle \omega,t\rangle} d\omega.
\end{alignat*}

This tutorial is organized as follows:
In Section \ref{sec:closer_look_continuous_shearlet_transform} we introduce the continuous shearlet transform
and prove the properties of the involved shearlets. 
We follow in Section \ref{sec:computation_discrete_shearlet_transform} the path via the continuous shearlet transform,
its counterpart on cones and finally its discretization on the full grid to obtain the
translation invariant discrete shearlet transform.
This is different to other implementations as, e.g., in ShearLab,
see \cite{KSZ12}.
Our discrete shearlet transform can be efficiently computed by the fast Fourier transform (FFT).
The discrete shearlets constitute a Parseval frame of the finite Euclidean space such that
the inversion of the shearlet transform can be simply done by applying the adjoint transform.
The second part of the section covers the implementation and installation details and provides
some performance measures.

\section{Closer Look at the Continuous Shearlet Transform in \texorpdfstring{$\RR^2$}{R 2}} 
\label{sec:closer_look_continuous_shearlet_transform}
In this section we combine some mostly well-known results from different authors. 
To make this tutorial self-contained and to obtain a complete documentation we also include the proofs. 
The functions are taken from \cite{Mey01,Mal08}. 
The construction of the shearlets is based on ideas from \cite{ELL08} and \cite{KLZ11}. 
The shearlet transform and the concept of shearlets on the cone were introduced in \cite{GKL06}.

\subsection{Some Functions and their Properties}
To define usable shearlets we need functions with special properties. 
We begin with defining these functions and prove their necessary properties. 
The results will be used later. 

We start by defining an auxiliary function $v\colon\RR\to\RR$ as
\begin{equation}\label{eq:v}
  v(x) 
:=
  \begin{cases}
    0                       & \text{for } x< 0,           \\
    35x^4-84x^5+70x^6-20x^7 & \text{for } 0\leq x \leq 1, \\
    1                       & \text{for } x>1.
  \end{cases}
\end{equation}
This function was proposed by Y. Meyer in \cite{Mey01,Mal08}, see Remark~\ref{remark:construction_v} for more information on the construction of $v$. 
Other choices of $v$ are possible, in \cite{KSZ12} the simpler function
\begin{equation*}
  \widetilde{v}(x) 
=
  \begin{cases}
    0          &\text{for } x< 0,                    \\
    2x^2       &\text{for } 0\leq x\leq\frac{1}{2},  \\
    1-2(1-x)^2 &\text{for } \frac{1}{2}\leq x\leq 1, \\
    1          &\text{for } x>1,
  \end{cases}
\end{equation*}
was chosen.

As we will see, the useful properties of $v$ for our purposes are its symmetry around $\left(\frac{1}{2},\frac{1}{2}\right)$ 
and the values at $0$ and $1$ with increase in between.
A plot of $v$ is shown in Figure~\ref{fig:meyerAux}.

Next we define the function $b:\RR\to\RR$ with
\begin{equation}\label{eq:b}
  b(\omega)
:=
  \begin{cases}
    \sin\left(\frac{\pi}{2}v(\lvert\omega\rvert-1)\right)                       & \text{for }1\leq\lvert\omega\rvert\leq 2, \\
    \cos\left(\frac{\pi}{2}v\left(\frac{1}{2}\lvert\omega\rvert-1\right)\right) & \text{for }2 < \lvert\omega\rvert\leq 4,  \\
    0                                                                           & \text{otherwise}.
  \end{cases}
\end{equation}
Note that $b$ is symmetric, positive, real-valued and $\supp b = [-4,-1]\cup[1,4]$. 
We further have that $b(\pm 2) = 1$. 
A plot of $b$ is shown in Figure~\ref{fig:meyerHelper}.

\begin{figure}[hptb]
  \centering
  \subfloat[$v(x)$.]{\label{fig:meyerAux}\includegraphics[width=0.4\textwidth]{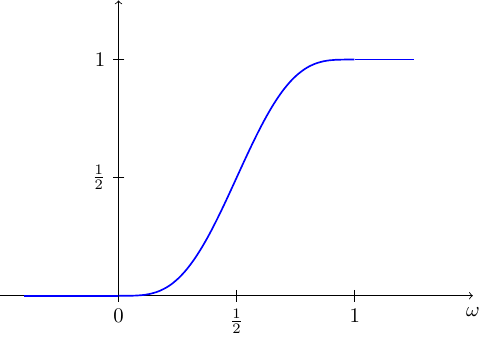}} 
  \hspace{1cm}
  \subfloat[solid: $b(\omega)$, dashed: $b(2\omega)$.]{
    \label{fig:meyerHelper}
    \raisebox{0.05in}{
      \includegraphics[height=0.27\textwidth]{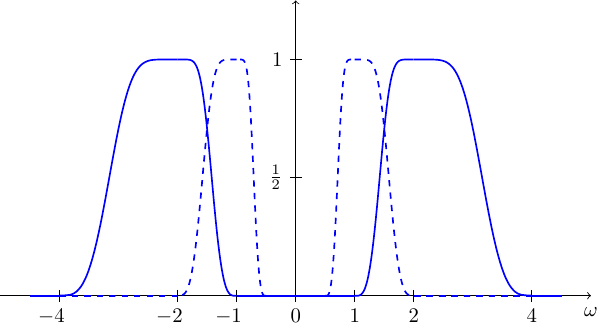}
    }
  }
  \caption{The two auxiliary functions $v$ (see \eqref{eq:v}) and $b$ (see \eqref{eq:b}).}
\end{figure}

Because of the symmetry we restrict ourselves in the following analysis to the case $\omega>0$.
Let $b_j:=b(2^{-j}\,\cdot\,)$, $j\in\NN_0$, thus, $\supp b_j = 2^j[1,4] = [2^{j},2^{j+2}]$ and $b_j(2^{j+1}) = 1$.
Observe that $b_j$ is increasing for $\omega\in [2^j,2^{j+1}]$ and decreasing for $\omega\in [2^{j+1},2^{j+2}]$.
Obviously all these properties carry over to $b_j^2$.
These facts are illustrated in the following diagram 
\begin{equation*}
\begin{array}{c|ccccccc}
  \omega  & 2^j &         & 2^{j+1} &          & 2^{j+2} &          & 2^{j+3} \\
  b_j     & 0   &\nearrow & 1       & \searrow & 0       &          &         \\
  b_{j+1} &     &         & 0       & \nearrow & 1       & \searrow & 0       \\
\end{array}
\end{equation*}
where $\nearrow$ stands for the increasing and $\searrow$ for the decreasing function.

For $j_1 \neq j_2$ the overlap between the support of $b_{j_1}^2$ and $b_{j_2}^2$ is empty except for $\lvert j_1-j_2\rvert = 1$.
Thus, for $b_j^2$ and $b_{j+1}^2$ we have that $\supp b_j^2 \cap \supp b_{j+1}^2 = [2^{j+1},2^{j+2}]$.
In this interval $b_j^2$ is decreasing with 
$
  b_j^2 
= 
  \cos^2\left(\frac{\pi}{2}v\left(\frac{2^{-j}}{2}\lvert\omega\rvert-1\right)\right)
$
and $b_{j+1}^2$ is increasing with 
$
  b_{j+1}^2 
= 
  \sin^2\left(\frac{\pi}{2}v(2^{-(j+1)}\lvert\omega\rvert-1)\right)
$.
Their sum in this interval is
\begin{equation*}
  b_j^2(\omega) + b_{j+1}^2(\omega)
=
  \cos^2\left(\frac{\pi}{2}v(2^{-j-1}\lvert\omega\rvert-1)\right)
  +
  \sin^2\left(\frac{\pi}{2}v(2^{-j-1}\lvert\omega\rvert-1)\right)
=
  1.
\end{equation*}

Hence, we can summarize
\begin{equation*}
  (b_j^2 + b_{j+1}^2)(\omega)
=
  \begin{cases}
    b_j^2     & \text{for }\omega< 2^{j+1},               \\
    1         & \text{for }2^{j+1}\leq\omega\leq 2^{j+2}, \\
    b_{j+1}^2 & \text{for }\omega>2^{j+2}.
  \end{cases}
\end{equation*}
Consequently, we have the following lemma.
\begin{lemma}\label{lemma:propertySumbj}
For $b_j$ defined as above, the relations
\begin{equation*}
  \sum_{j=-1}^\infty 
  b_j^2(\omega) 
= 
  \sum_{j=-1}^\infty 
  b^2(2^{-j}\omega) 
= 
  1
\quad\text{for }
  \lvert\omega\rvert\geq 1
\end{equation*}
and
\begin{equation}\label{eq:OverlapPsi1}
  \sum_{j=-1}^\infty b_j^2(\omega)
=
\begin{cases}
  0                                            & \text{for }\lvert\omega\rvert\leq\frac{1}{2}, \\
  \sin^2\left(\frac{\pi}{2}v(2\omega-1)\right) & \text{for }\frac{1}{2}<\lvert\omega\rvert<1,  \\
  1                                            & \text{for }\lvert\omega\rvert\geq 1
\end{cases}
\end{equation}
hold true.
\end{lemma}
\begin{proof}
In each interval $[2^{j+1},2^{j+2}]$ only $b_j$ and $b_{j+1}$, $j\geq -1$, are not equal to zero. Thus, it is sufficient to prove that $b_j^2+b_{j+1}^2\equiv 1$ in this interval. We get that
\begin{align*}
  (b_j^2+b_{j+1}^2)(\omega)
&=
  \quad
  b^2(
  \underbrace{
  2^{-j}\omega
  }_{
  \makebox[0mm][c]{ 
    \begin{tiny}$\in 2^{-j}[2^{j+1},2^{j+2}] = [2,4]$\end{tiny}
  }
  }
  )
  \qquad+\qquad
  b^2(
  \underbrace{
  2^{-j-1}\omega
  }_{
  \makebox[0mm][c]{ 
    \begin{tiny}
    $\in 2^{-j-1}[2^{j+1},2^{j+2}] = [1,2]$
    \end{tiny}
  }
  }
)\\
&=
  \cos^2\left(\frac{\pi}{2}v\left(\frac{1}{2}\cdot 2^{-j}\omega-1\right)\right)
  +
  \sin^2\left(\frac{\pi}{2}v(2^{-j-1}\omega-1)\right)
  \\
&=
  \cos^2\left(\frac{\pi}{2}v(2^{-j-1}\omega-1)\right)
  +
  \sin^2\left(\frac{\pi}{2}v(2^{-j-1}\omega-1)\right)
  \\
&=
  1.
\end{align*}
The second relation follows by straightforward computation.
\end{proof}

We define the function $\psi_1\colon\RR\to\RR$ via its Fourier transform as
\begin{equation}\label{eq:Psi1}
  \hat{\psi}_1(\omega) := \sqrt{b^2(2\omega) + b^2(\omega)}.
\end{equation}
Figure~\ref{fig:Psi1} shows the function $\hat{\psi}_1$. 
The following theorem states an important property of $\psi_1$.
\begin{theorem}\label{thm:PropertyPsi1}
  The above defined function $\hat{\psi}_1$ has 
  $\supp\hat{\psi}_1 = [-4,-\frac{1}{2}]\cup[\frac{1}{2},4]$ 
  and fulfills
  \begin{equation*}
    \sum_{j\geq 0}
    \lvert \hat{\psi}_1(2^{-2j}\omega)\rvert^2 
  = 
    1\quad\text{for }\lvert\omega\rvert>1.
  \end{equation*}
\end{theorem}
%
\begin{proof}
The assumption on the support follows from the definition of $b$. For the sum we have
\begin{equation*}
  \sum_{j\geq 0}\lvert\hat{\psi}_1(2^{-2j}\omega)\rvert^2
=
  \sum_{j=0}^\infty
  b^2(2\cdot 2^{-2j}\omega) + b^2(2^{-2j}\omega)
=
  \sum_{j=0}^\infty
  b^2(
    2^{-2j+1}
    \omega
  )
  +
  b^2(
    2^{-2j}\omega
  )
\end{equation*}
where $-2j+1 \in \{+1, -1, -3,\ldots\}$ (odd) and $-2j \in \{ 0,-2,-4,\ldots \}$ (even).
Thus, by Lemma~\ref{lemma:propertySumbj}, we get
\begin{equation*}
  \sum_{j\geq 0}\lvert\hat{\psi}_1(2^{-2j}\omega)\rvert^2
=
  \sum_{j=-1}^\infty
  b^2(2^{-j}\omega)
=
  1.
  \qedhere
\end{equation*}
\end{proof}

By \eqref{eq:OverlapPsi1} we have that
\begin{equation}\label{eq:OverlapPsi1_2}
 \sum_{j\geq 0}
  \lvert\hat{\psi}_1(2^{-2j}\omega)\rvert^2
=
\begin{cases}
  0                                            &\text{for }\lvert\omega\rvert\leq\frac{1}{2}, \\
  \sin^2\left(\frac{\pi}{2}v(2\omega-1)\right) &\text{for }\frac{1}{2}<\lvert\omega\rvert<1,  \\
  1                                            &\text{for }\lvert\omega\rvert\geq 1.
\end{cases}
\end{equation}

Next we define a second function $\psi_2\colon\RR\to\RR$---again in Fourier domain---by
\begin{equation}\label{eq:Psi2}
  \hat{\psi}_2(\omega) 
:=
  \begin{cases}
    \sqrt{v(1+\omega)}&\text{for }\omega\leq 0, \\
    \sqrt{v(1-\omega)}&\text{for }\omega >0.
  \end{cases}
\end{equation}
The function $\hat{\psi}_2$ is shown in Figure~\ref{fig:Psi2}. 
Before stating a theorem about the properties of $\hat{\psi}_2$ we need the following two auxiliary lemmas.
\begin{figure}[htbp]
  \centering
  \subfloat[$\hat{\psi}_1$.]{\label{fig:Psi1}
  \raisebox{0.03in}{\includegraphics[width=0.4\textwidth]{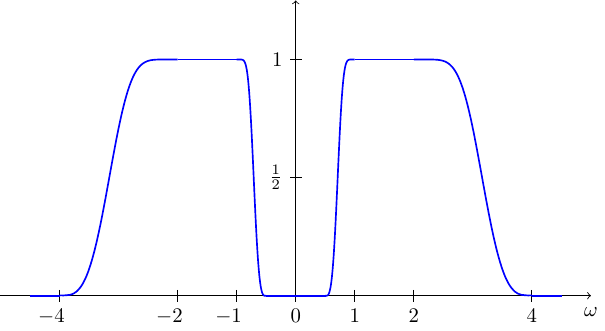}}}
  \hspace{1cm}
  \subfloat[$\hat{\psi}_2$.]{\label{fig:Psi2}\includegraphics[height=0.22\textwidth]{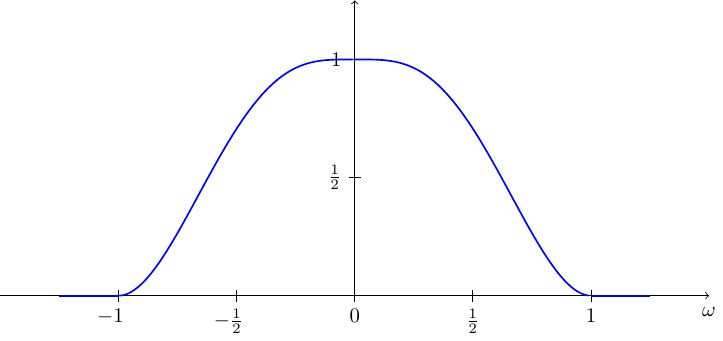}}
  \caption{The functions $\hat{\psi}_1$ (see \eqref{eq:Psi1}) and $\hat{\psi}_2$ (see \eqref{eq:Psi2}).}
\end{figure}
Recall that a function $f\colon \RR\to\RR$ is \emph{point symmetric with respect to $(a,b)$} if and only if
\begin{equation*}
  f(a+x) - b 
= 
  -f(a-x) + b 
\quad\text{for all }
  x\in\RR.
\end{equation*}
With the substitution $x+a \to x$ this is equivalent to
\begin{equation*}
  f(x) + f(2a-x) 
= 
  2b \quad\text{for all } x\in\RR.
\end{equation*}
Thus, for a function symmetric to $(0.5,0.5)$ we have that $f(x) + f(1-x) = 1$.
\begin{lemma}
  The function $v$ in \eqref{eq:v} is symmetric with respect to 
  $(0.5,0.5)$, in particular, $v(x)+v(1-x) = 1$ for all $x\in\RR$.
\end{lemma}
\begin{proof}
The symmetry is obvious for $x<0$ and $x>1$. 
It remains to prove the symmetry for $0\leq x \leq 1$. 
We will see in Remark~\ref{remark:construction_v} that $v'(x) = -140 x^3(x-1)^3$.
With the fundamental theorem of calculus this implies that 
\begin{equation*}
  v(x) 
= 
  -140 \int_0^x t^3(t-1)^3 \dif{t}.
\end{equation*}
Since $v(1) = 1$ we know that $v(1) = -140 \int_0^1 t^3(t-1)^3 \dif{t} = 1$.
Next, consider
$
  v(1-x) 
= 
  -140\int_0^{1-x} t^3(t-1)^3 \dif{t}
$.
Substituting $t\to 1-t$ yields
\begin{equation*}
  v(1-x) 
= 
  140 \int_1^x(1-t^3)(-t)^3 \dif{t} 
= 
  -140 \int_x^1 t^3(t-1)^3\dif{t}.
\end{equation*}
Finally, we obtain for the sum
\begin{equation*}
  v(x) + v(1-x)
=
  -140 \int_0^x t^3(t-1)^3 \dif{t}
  -140 \int_x^1 t^3(t-1)^3\dif{t}
=
  -140 \int_0^1 t^3(t-1)^3\dif{t}
=
  1.
  \qedhere
\end{equation*}
\end{proof}
Note that $\hat{\psi}_2$ is axially symmetric to the $y$-axis.

\begin{lemma}\label{lemma:PropertyPsi2}
  The function $\hat{\psi}_2$ fulfills
  \begin{equation*}
    \hat{\psi}_2^2(\omega-1) 
  + 
    \hat{\psi}^2_2(\omega) 
  + 
    \hat{\psi}^2_2(\omega+1) 
  = 
    1
  \quad\text{for }
    \lvert\omega\rvert\leq 1.
  \end{equation*}
\end{lemma}
\begin{proof}
We have
\begin{equation*}
  \hat{\psi}^2_2(\omega) 
=
  \begin{cases}
    v(1+\omega) & \text{for }\omega\leq 0, \\
    v(1-\omega) & \text{for }\omega >0.
  \end{cases}
\end{equation*}
Consequently, we get for $0\leq\omega\leq 1$ that
\begin{align*}
  \hat{\psi}^2_2(\omega-1) + \hat{\psi}^2_2(\omega) + \hat{\psi}^2_2(\omega+1)
&=
  v(1+\omega-1) + v(1-\omega) + v(1-\omega-1)
\\
&=
  v(\omega) + v(1-\omega) + \underbrace{v(-\omega)}_{=0}
=
  1,
\end{align*}
and for $-1\leq\omega<0$  we obtain similarly
\begin{align*}
  \hat{\psi}^2_2(\omega-1) + \hat{\psi}^2_2(\omega) + \hat{\psi}^2_2(\omega+1)
&=
  v(1+\omega-1) + v(1-\omega) + v(1-\omega-1)
\\
&=
  \underbrace{v(-\lvert\omega\rvert)}_{=0} + v(1-\lvert\omega\rvert) + v(\lvert\omega\rvert)
=
  1.
  \qedhere
\end{align*}
\end{proof}
It can be seen in the proof that the sum reduces in both cases to two (different) summands, 
in particular
\begin{equation*}
  1
=
  \hat{\psi}_2^2(\omega-1) + \hat{\psi}^2_2(\omega) + \hat{\psi}^2_2(\omega+1)
=
  \begin{cases}
    \hat{\psi}_2^2(\omega-1) + \hat{\psi}^2_2(\omega)   & \text{for }  0 \leq \omega \leq 1, \\
    \hat{\psi}^2_2(\omega)   + \hat{\psi}^2_2(\omega+1) & \text{for } -1 \leq \omega <    0.
  \end{cases}
\end{equation*}

With these lemmas we can prove the next theorem.
\begin{theorem}\label{thm:PropertyPsi2}
  The function $\hat{\psi}_2$ defined in \eqref{eq:Psi2} fulfills
  \begin{equation}\label{eq:PropertyPsi2}
    \sum_{k=-2^j}^{2^j}\lvert\hat{\psi}_2(k+2^j\omega)\rvert^2
  = 
    1
  \quad\text{for }
    \lvert\omega\rvert\leq 1,\ j\geq 0.
  \end{equation}
\end{theorem}
\begin{proof}
With $\widetilde{\omega} := 2^j\omega$ the assertion in \eqref{eq:PropertyPsi2} becomes
\begin{equation*}
  \sum_{k=-2^j}^{2^j}\lvert\hat{\psi}_2(k+\tilde{\omega})\rvert^2
= 
  1
\quad\text{for }
  \lvert\tilde{\omega}\rvert\leq 2^j,\ j\geq 0.
\end{equation*}
For a fixed (but arbitrary) 
$\omega^\star\in[-2^j,2^j]\subset\RR$
we need 
$-1 \leq \omega^\star + k \leq 1$ 
for 
$\hat{\psi}_2(\omega^\star+k)\neq 0$
since $\supp\hat{\psi}_2 = [-1,1]$.
Thus, for $\omega^\star\in\ZZ$, 
only the summands for
$k\in\{-\omega^\star-1,-\omega^\star,-\omega^\star+1\}$ 
do not vanish.
But for $k=-\omega^\star\pm 1$ we have $\omega^\star + k = \pm 1$ and $\hat{\psi}_2(\pm 1) = 0$.
In this case the entire sum reduces to one summand $k=-\omega^\star$ such that
\begin{equation*}
  \sum_{k=-2^j}^{2^j}\lvert\hat{\psi}_2(k+\omega^\star)\rvert^2
=
  \lvert\hat{\psi}_2(-\omega^\star+\omega^\star)\rvert^2
=
  \lvert\hat{\psi}_2(0)\rvert^2
=
  1.
\end{equation*}
If $\omega^\star\not\in\ZZ$ and $\omega^\star>0$ the only non-zero summands appear for $k\in\{\lfloor\omega^\star\rfloor,\lfloor\omega^\star\rfloor-1\}$.
Thus, 
$0<r^+ := \omega^\star-\lfloor\omega^\star\rfloor<1$,
yields
\begin{equation*}
  \sum_{k=-2^j}^{2^j}\lvert\hat{\psi}_2(k+\omega^\star)\rvert^2
=
  \lvert\hat{\psi}_2(-\lfloor\omega^\star\rfloor+\omega^\star)\rvert^2
  +
  \lvert\hat{\psi}_2(-\lfloor\omega^\star\rfloor-1+\omega^\star)\rvert^2
=
  \lvert\hat{\psi}_2(r^+)\rvert^2 + \lvert\hat{\psi}_2(1-r^+)\rvert^2
\end{equation*}
which is equal to $1$ by Lemma~\ref{lemma:PropertyPsi2}.
Analogously we obtain for $\omega^\star\not\in\ZZ$, $\omega^\star<0$
that the remaining non-zero summands are those for
$
  k
\in
  \{
    \lceil\omega^\star\rceil,
    \lceil\omega^\star\rceil+1
\}$. 
With
$-1<r^-:= \lceil\omega^\star\rceil+\omega^\star<0$ we get
\begin{equation*}
  \sum_{k=-2^j}^{2^j}\lvert\hat{\psi}_2(k+\omega^\star)\rvert^2
=
  \lvert\hat{\psi}_2(\lceil\omega^\star\rceil+\omega^\star)\rvert^2
  +
  \lvert\hat{\psi}_2(\lceil\omega^\star\rceil+1+\omega^\star)\rvert^2
=
  \lvert\hat{\psi}_2(r^-)\rvert^2 + \lvert\hat{\psi}_2(1+r^-)\rvert^2.
\end{equation*}
By Lemma~\ref{lemma:PropertyPsi2} and since 
$\hat{\psi}_2(x) = \hat{\psi}_2(-x)$, 
we finally conclude
\begin{equation*}
  \lvert\hat{\psi}_2(r^-)\rvert^2 + \lvert\hat{\psi}_2(1+r^-)\rvert^2
=
  \lvert\hat{\psi}_2(\lvert r^-\rvert)\rvert^2 + \lvert\hat{\psi}_2(1-\lvert r^-\rvert)\rvert^2
=
  1.
  \qedhere
\end{equation*}
\end{proof}

\subsection{The Continuous Shearlet Transform}
For the shearlet transform we use the dilation matrix $A_a$ and shear matrix $S_s$
For $d=2$ and $\gamma=\frac{1}{2}$ they read
\begin{equation}\label{eq:dilationAndShearMatrix}
  A_a
=
  \begin{pmatrix} a & 0 \\ 0 & \sqrt{a}\end{pmatrix},
\quad
  a\in\RR^+,
\qquad
  S_s
=
  \begin{pmatrix} 1 & s \\ 0 & 1\end{pmatrix},
\quad
  s\in\RR.
\end{equation}
The shearlets $\psi_{a,s,t}$ emerge by dilation, shear and translation of a function 
$\psi\in L_2(\RR^2)$ as before
\begin{equation}\label{eq:psi_ast}
  \psi_{a,s,t}(x)
=
  a^{-\frac{3}{4}}\psi(A_a^{-1}S_s^{-1}(x-t))
=
  a^{-\frac{3}{4}}
  \psi\left(
    \begin{pmatrix}
      \frac{1}{a} & -\frac{s}{a}       \\ 
      0           & \frac{1}{\sqrt{a}}
    \end{pmatrix}
    (x-t)
  \right).
\end{equation}
We assume that $\hat{\psi}$ can be written as 
\begin{equation}\label{eq:shearlet_tensor_product}\index{shearlet}
  \hat{\psi}(\omega_1,\omega_2)
=
  \hat{\psi_1}(\omega_1)\hat{\psi_2}\left(\frac{\omega_2}{\omega_1}\right).
\end{equation} 
Consequently, we obtain for the Fourier transform
\begin{align*}
  \hat{\psi}_{a,s,t}(\omega)
&=
  a^{-\frac{3}{4}}
  \cF\left(
  \psi
  \left(
    \begin{pmatrix} 
      \frac{1}{a} & -\frac{s}{a}       \\ 
      0           & \frac{1}{\sqrt{a}}
    \end{pmatrix}
    (\cdot-t)
  \right)
  \right)(\omega)
\\
&=
  a^{-\frac{3}{4}} 
  \e^{-2\pi i\langle \omega,t \rangle} 
  \cF\left(
  \psi\left(
    \begin{pmatrix}
      \frac{1}{a} & -\frac{s}{a}       \\ 
      0           & \frac{1}{\sqrt{a}}
    \end{pmatrix}
    \cdot
  \right)
  \right)(\omega)
\\
&=
  a^{-\frac{3}{4}} 
  \e^{-2\pi i\langle \omega,t \rangle} 
  (a^{-\frac{3}{2}})^{-1}
  \hat{\psi}
  \left(
    \begin{pmatrix} 
      a         &        0 \\ 
      s\sqrt{a} & \sqrt{a}
    \end{pmatrix}
    \omega
  \right)
\\
&=
  a^{\frac{3}{4}} 
  \e^{-2\pi i\langle \omega,t \rangle} 
  \hat{\psi}
  \left(
    a\omega_1,\sqrt{a}(s\omega_1+\omega_2)
  \right)
\\
&=
  a^{\frac{3}{4}} 
  \e^{-2\pi i\langle \omega,t \rangle} 
  \hat{\psi}_1
  \left(
    a\omega_1
  \right)
  \hat{\psi}_2
  \left(
    a^{-\frac{1}{2}}
    \left(
      \frac{\omega_2}{\omega_1}+s
    \right)
  \right).
\end{align*}

The shearlet transform\index{shearlet transform} $\cSH_\psi (f)$ of $f\in L^2(\RR)$ is given as
\begin{align*}
  \cSH _\psi (f)(a,s,t)
&=
  \langle f,\psi_{a,s,t}\rangle
\\
&=
  \langle \hat{f},\hat{\psi}_{a,s,t} \rangle
\\
&=
  \int_{\RR^2} 
    \hat{f}(\omega) \overline{\hat{\psi}_{a,s,t}(\omega)} 
  \dif{\omega}
  \\
&=
  a^{\frac{3}{4}}
  \int_{\RR^2}
  \hat{f}(\omega)
  \hat{\psi}_1(a\omega_1)
  \hat{\psi}_2\left(a^{-\frac{1}{2}}\left(\frac{\omega_2}{\omega_1}+s\right)\right)
  \e^{2\pi i\langle \omega,t \rangle}
  \dif{\omega}
  \\
&=
  a^{\frac{3}{4}}
  \cF ^{-1}
  \left(
  \hat{f}(\omega)
  \hat{\psi}_1(a\omega_1)
  \hat{\psi}_2\left(a^{-\frac{1}{2}}\left(\frac{\omega_2}{\omega_1}+s\right)\right)
  \right) (t).
\end{align*}
The same formula is derived by interpreting the shearlet transform as a convolution with the function $\psi_{a,s}(x) = \overline{\psi}(-A_a^{-1}S_{s}^{-1}x)$ and using the convolution theorem.

The shearlet transform is invertible if the function $\psi$ fulfills the \emph{admissibility property}
\begin{equation*}
  \int_{\RR^2} 
  \frac{\lvert\hat \psi(\omega_1,\omega_2)\rvert^2}
  {\lvert\omega_1\rvert^2} 
  \dif{\omega_1}\dif{\omega_2} 
< 
  \infty.
\end{equation*}
Easy calculations show that any shearlet of the form \eqref{eq:shearlet_tensor_product},
where $\hat{\psi}_1$ and $\hat{\psi}_2$ are continuous and 
$\supp(\hat{\psi}_1)\subset [-b,-a]\cup [a,b]$ and $\supp(\hat{\psi}_2)\subset [-c,c]$, $a,b,c > 0$, 
is admissible.
Figure~\ref{fig:shearletFourierTime} shows a dilated and sheared shearlet in Fourier and time domain.

\begin{figure}[htbp]
  \centering
  \subfloat[Shearlet in Fourier domain \newline for $a=\frac{1}{4}$ and $s=-\frac{1}{2}$.]{\label{fig:shearletFourier}
  \imageWithBorder{0.28\textwidth}{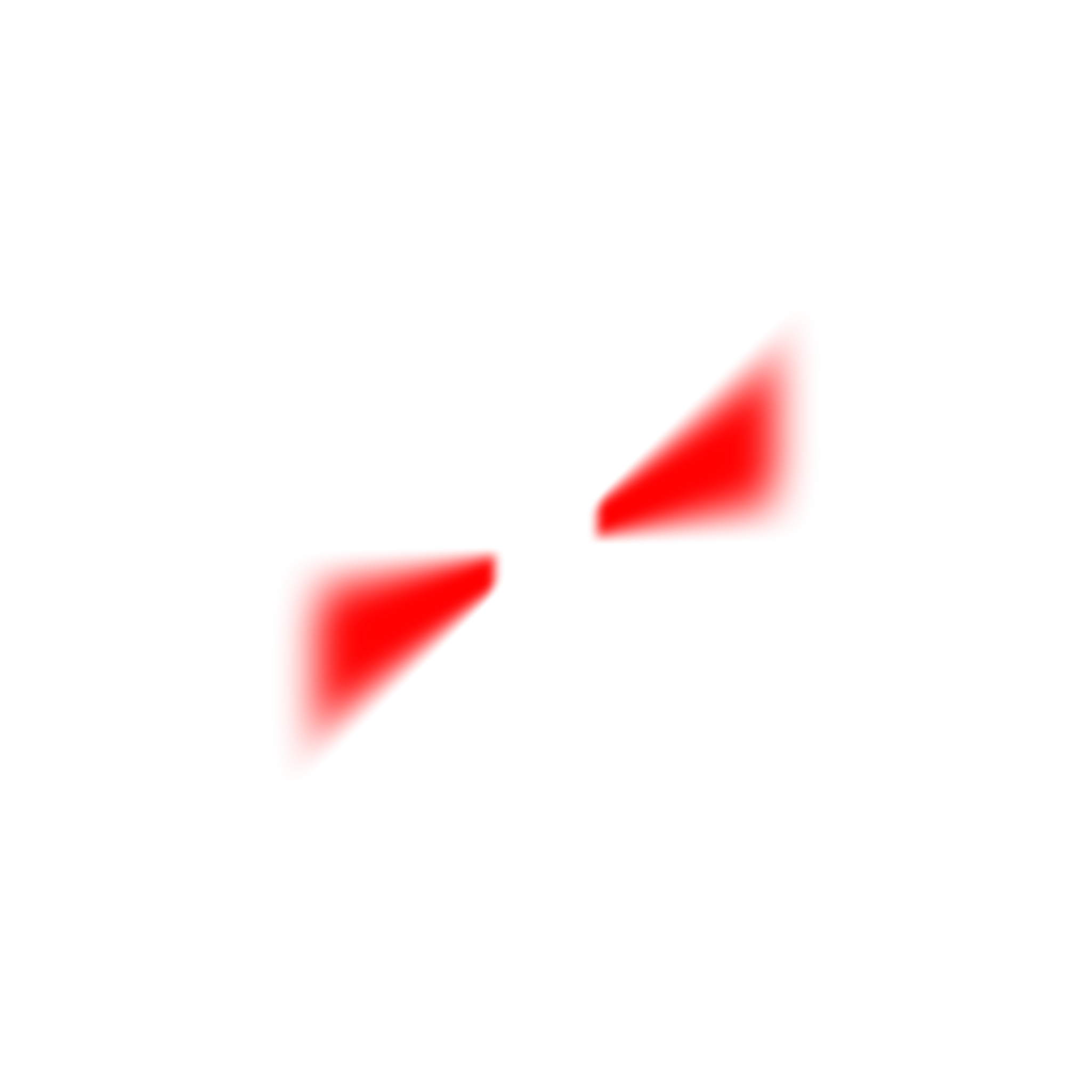}}
  \hspace{0.2cm}
  \subfloat[Same shearlet in time domain (zoomed).]{\label{fig:shearletTime}
  \imageWithBorder{0.28\textwidth}{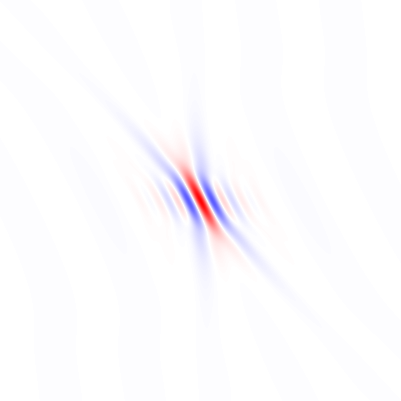}}
  \hspace{0.2cm}
  \subfloat[3D-view of time domain shearlet.]{\label{fig:shearlet_time_3D}
  \imageWithBorder{0.28\textwidth}{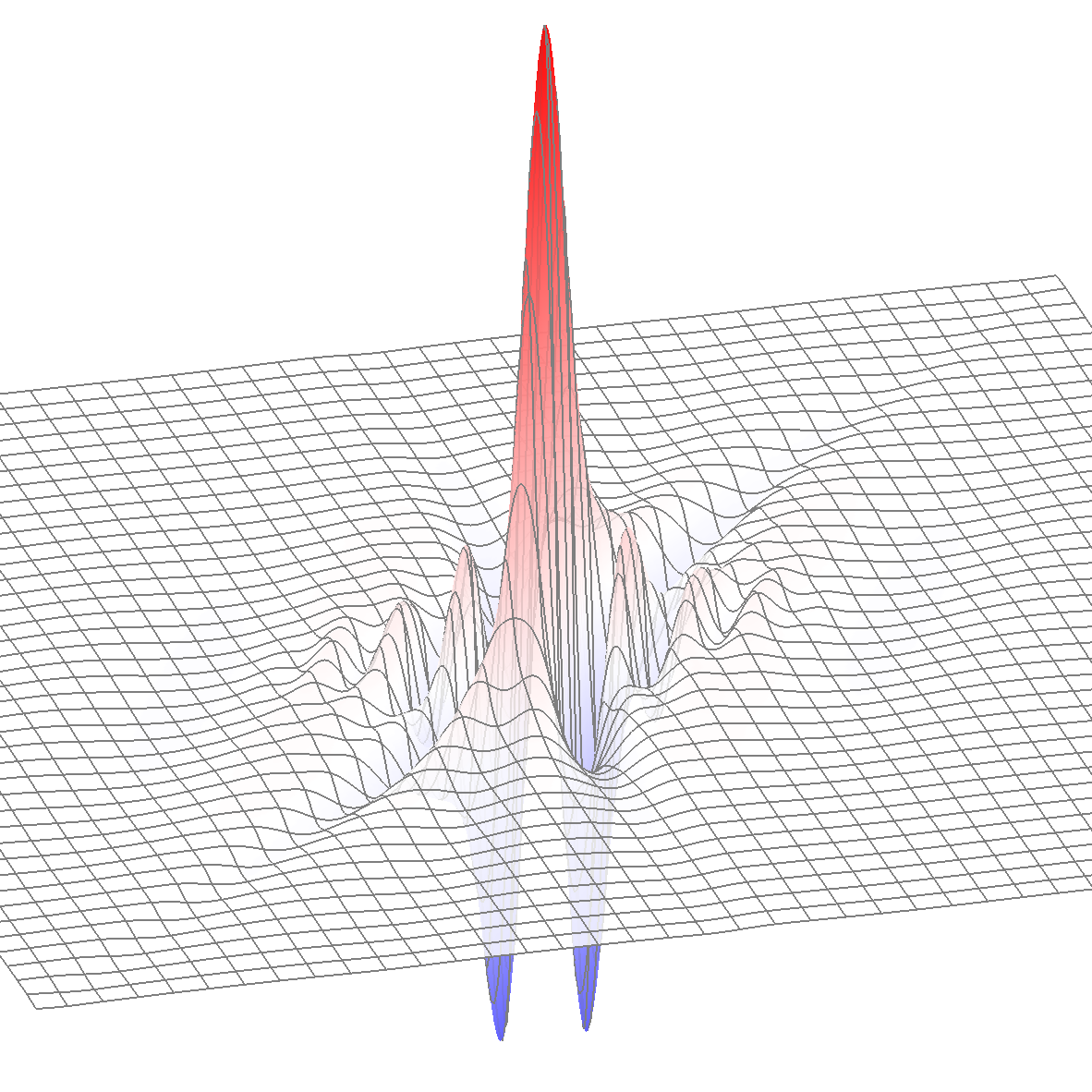}}
  \caption{Shearlet in Fourier and time domain.}
  \label{fig:shearletFourierTime}
\end{figure}

\subsection{Shearlets on the Cone}
Up to now we have said nothing about the support of our shearlet $\psi$.
We use band-limited shearlets, thus, we have compact support in Fourier domain.
In the previous section we assumed that 
$
  \hat{\psi}(\omega_1,\omega_2) 
= 
  \hat{\psi_1}(\omega_1)\hat{\psi}_2\left(\frac{\omega_2}{\omega_1}\right)
$,
where we now define $\psi_1$ and $\psi_2$ as in \eqref{eq:Psi1} and \eqref{eq:Psi2}, respectively. 
With the results shown for $\hat{\psi}_1$ for $\lvert\omega_1\rvert\geq\frac{1}{2}$ and 
$\hat{\psi}_2$ for $\lvert\omega\rvert<1$, i.e., $\lvert\omega_2\rvert<\lvert\omega_1\rvert$, 
it is natural to define the area
\begin{equation*}
    \cC ^h
  :=
    \left\{
      (\omega_1,\omega_2)\in\RR^2 
    : 
      \lvert\omega_1\rvert \geq \tfrac{1}{2}, 
      \lvert\omega_2\rvert<\lvert\omega_1\rvert
    \right\}.
\end{equation*}
We will refer to this set as the \emph{horizontal cone} (see Figure~\ref{fig:cones}). Analogously we define the \emph{vertical cone} as
\begin{equation*}
    \cC ^v
  :=
    \left\{
      (\omega_1,\omega_2)\in\RR^2 
    : 
      \lvert\omega_2\rvert \geq \tfrac{1}{2}, 
      \lvert\omega_2\rvert>\lvert\omega_1\rvert
    \right\}.
\end{equation*}
To cover the entire $\RR^2$ we define two more sets
\begin{align*}
    \cC ^\times
  &:=
    \left\{
      (\omega_1,\omega_2)\in\RR^2 
    : 
      \lvert\omega_1\rvert \geq \tfrac{1}{2}, 
      \lvert\omega_2\rvert \geq \tfrac{1}{2}, 
      \lvert\omega_1\rvert=\lvert\omega_2\rvert
    \right\},
  \\
    \cC ^0
  &:=
    \left\{
        (\omega_1,\omega_2)\in\RR^2 
      : 
        \lvert\omega_1\rvert < 1, 
        \lvert\omega_2\rvert < 1
    \right\},
\end{align*}
where $\cC ^\times$ is the intersection (or the seam lines) of the two cones and 
$\cC ^0$ is the ``low-frequency'' part.
Altogether we have
$
 \RR^2
=
 \cC ^h \cup
 \cC ^v \cup
 \cC ^\times \cup
 \cC ^0
$
with an overlapping domain
\begin{equation}\label{eq:overlappingDomain}
  \cone^\square
:=
  (-1,1)^2\setminus(-\tfrac{1}{2},\tfrac{1}{2})^2.
\end{equation}

\begin{figure}[htbp]
  \centering
  \includegraphics[width=0.4\textwidth]{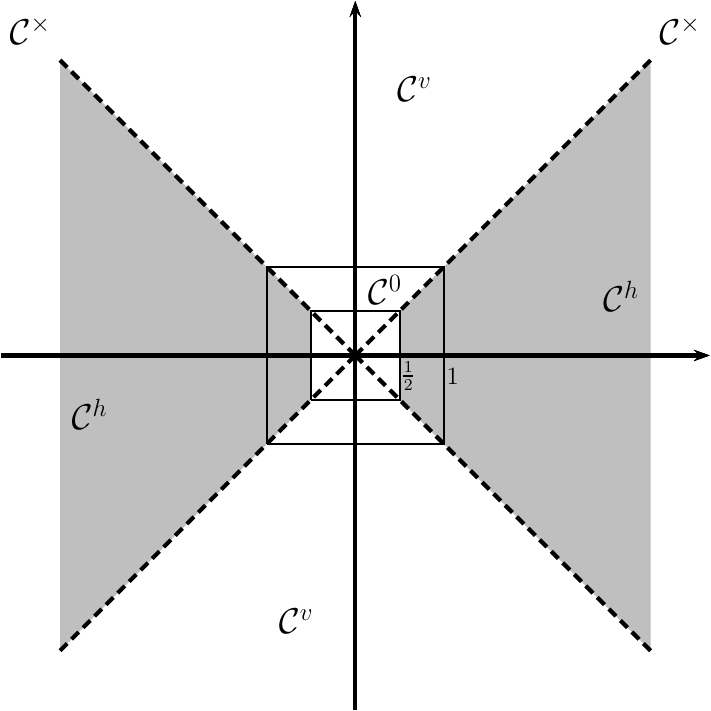}
  \caption{The sets $\cone^h$, $\cone^v$, $\cone^\times$ and $\cone^0$.}
  \label{fig:cones}
\end{figure}

Obviously the shearlet $\psi$ defined above is perfectly suited for the horizontal cone.
For each set $\cone^\kappa$, $\kappa \in \{h,v,\times\}$, we define a characteristic function
$\chi_{\cone^\kappa}(\omega)$ which is equal to 1 for $\omega\in\cone^\kappa$ and 0 for $\omega\not\in\cone^\kappa$.
We need these characteristic functions as cut-off functions at the seam lines. We set
\begin{equation}\label{eq:horizontalShearlet}\index{shearlet!on the cone}
  \hat{\psi}^h(\omega_1,\omega_2)
:=
  \hat{\psi}(\omega_1,\omega_2)
=
  \hat{\psi}_1(\omega_1)
  \hat{\psi}_2\left(\frac{\omega_2}{\omega_1}\right)
  \chi_{\cone^h}.
\end{equation}
For the non-dilated and non-sheared $\hat{\psi}^h$ the cut-off function has no effect
since the support of $\hat{\psi}^h$ is completely contained in $\cone^h$.
But after the dilation and shear we have
\begin{equation*}
  \supp \hat{\psi}_{a,s,0}
\subseteq
  \left\{
  (\omega_1,\omega_2) :
  \frac{1}{2a}\leq\left\lvert\omega_1\right\rvert\leq\frac{4}{a}, \left\lvert s + \frac{\omega_2}{\omega_1}\right\rvert\leq\sqrt{a}
  \right\}.
\end{equation*}
The question arises for which $a$ and $s$ this set remains a subset of the horizontal cone.
For $a>1$ we have that $\omega_1\leq\frac{1}{2}$ is in $\supp \hat{\psi}_{a,s,0}$ 
but not in $\cone^h$.
Thus, we can restrict ourselves to $a\leq 1$.

With $a$ fixed the second condition for $\supp \hat{\psi}_{a,s,0}$ reads
\begin{alignat}{3}\notag
 -\sqrt{a}
&\leq
  s + \frac{\omega_1}{\omega_2}
&\leq&
  \sqrt{a},
\\\label{eq:condition_s}
 -\sqrt{a} - s
&\leq\quad
  \frac{\omega_1}{\omega_2}
  \quad
&\leq&
  \sqrt{a} - s.
\end{alignat}
Since $\left\lvert\frac{\omega_1}{\omega_2}\right\rvert\leq 1$
the right condition becomes $\sqrt{a}-s\leq 1$ and
for the left condition $-\sqrt{a}-s\geq -1$,
hence, we can conclude
\begin{equation*}
 -1+\sqrt{a} 
\leq 
  s 
\leq 
  1 - \sqrt{a}.
\end{equation*}
For such $s$ it holds that $\supp\hat{\psi}_{a,s,0}\subseteq\cone^h$,
in particular the indicator function is not needed for these $s$ (with respective $a$).
One might ask for which $s$ (depending on $a$) the indicator function cuts off only parts of the function,
i.e., $\supp \hat{\psi}_{a,s,0}\cap\cone^h \neq \emptyset$.
We take again \eqref{eq:condition_s} but now we do not use a condition to \emph{guarantee} that 
$\left\lvert\frac{\omega_1}{\omega_2}\right\rvert \leq 1$
but rather ask for a condition that \emph{allows} 
$\left\lvert\frac{\omega_1}{\omega_2}\right\rvert \leq 1$.
Thus, the right bound $\sqrt{a}-s$ should be larger than $-1$ and the left bound $-\sqrt{a}-s$ should be smaller than $1$. 
Consequently, we obtain
\begin{equation*}
  -1-\sqrt{a} \leq s \leq 1 + \sqrt{a}.
\end{equation*}
Summing up, we have for $\lvert s\rvert\leq 1 - \sqrt{a}$ that $\supp\hat{\psi}_{a,s,0}\subseteq\cone^h$.
For $1-\sqrt{a}<\lvert s\rvert<1+\sqrt{a}$ parts of $\supp\hat{\psi}_{a,s,0}$ are also in $\cone^v$,
which are cut off.
For $\lvert s\rvert>1+\sqrt{a}$ the whole shearlet is set to zero by the characteristic function.
If we go back to the definition of $\hat{\psi}_{a,s,0}$ we see that the vertical range is determined by $\hat{\psi}_2$.
By definition $\hat{\psi}_2$ is axially symmetric with respect to the $y$-axis, 
in other words the ``center'' of $\hat{\psi}_2$ is taken for the argument equal to zero, i.e., $a^{-\frac{1}{2}}\left(\frac{\omega_1}{\omega_2}+s\right)=0$.
It follows that for $\lvert s\rvert=1$ the center of $\hat{\psi}_{a,s,0}$ is at the seam-lines. 
Thus, for $\lvert s\rvert=1$ approximately one half of the shearlet is cut off whereas the other half remains.
For larger $s$ larger parts would be cut. 
Consequently, we restrict ourselves to $\lvert s\rvert\leq 1$.

The shearlet for the vertical cone is defined analogously with the roles of $\omega_1$ and $\omega_2$ interchanged, i.e.,
\begin{equation}\label{eq:verticalShearlet}
  \hat{\psi}^v(\omega_1,\omega_2)
:=
  \hat{\psi}(\omega_2,\omega_1)
=
  \hat{\psi}_1(\omega_2)
  \hat{\psi}_2\left(\frac{\omega_1}{\omega_2}\right)
  \chi_{\cone^v}.
\end{equation}
All the results from above apply to this setting.
For $(\omega_1,\omega_2)\in\cone^\times$, i.e., $\lvert\omega_1\rvert=\lvert\omega_2\rvert$, both definitions coincide and we define
\begin{equation}\label{eq:seamlineShearlet}
  \hat{\psi}^\times(\omega_1,\omega_2)
:=
  \hat{\psi}(\omega_1,\omega_2)\chi_{\cone^\times}.
\end{equation}
The shearlets $\hat{\psi}_h$, $\hat{\psi}_v$ (and $\hat{\psi^\times}$) are called \emph{shearlets on the cone}. 
This concept was introduced in \cite{GKL06}.

We have functions to cover three of the four parts of $\RR^2$. 
The remaining part $\cone^0$ will be handled with a scaling function which is presented in the next section.

\subsection{Scaling Function}\label{sec:scaling_function}
For the center part $\cone^0$ (also low-frequency part) we define another set of functions.
To this end, we need the following \emph{scaling function}\index{scaling function}
\begin{equation*}
  \varphi(\omega) 
:=
  \begin{cases}
    1                                                      & \text{for } \lvert\omega\rvert\leq\frac{1}{2}, \\
    \cos\left(\frac{\pi}{2}v(2\lvert\omega\rvert-1)\right) & \text{for } \frac{1}{2}<\lvert\omega\rvert<1,  \\
    0                                                      & \text{otherwise}.
  \end{cases}
\end{equation*}
The full \emph{scaling function} $\phi$ is then defined as
\begin{align}\label{eq:phi}
  \hat{\phi}(\omega_1,\omega_2)
:=&
  \begin{cases}
    \varphi(\omega_1) & \text{for } \lvert\omega_2\rvert\leq\lvert\omega_1\rvert, \\
    \varphi(\omega_2) & \text{for } \lvert\omega_1\rvert<\lvert\omega_2\rvert     \\
  \end{cases}
\\\notag
=&
  \begin{cases}
    1                                                        & \text{for } \lvert\omega_1\rvert\leq\frac{1}{2},\lvert\omega_2\rvert\leq\frac{1}{2},            \\
    \cos\left(\frac{\pi}{2}v(2\lvert\omega_1\rvert-1)\right) & \text{for } \frac{1}{2}<\lvert\omega_1\rvert < 1, \lvert\omega_2\rvert\leq\lvert\omega_1\rvert, \\ 
    \cos\left(\frac{\pi}{2}v(2\lvert\omega_2\rvert-1)\right) & \text{for } \frac{1}{2}<\lvert\omega_2\rvert < 1, \lvert\omega_1\rvert<\lvert\omega_2\rvert,    \\
    0                                                        & \text{otherwise}.
  \end{cases}
\end{align}
The decay of the scaling function $\hat{\phi}$ (respectively $\varphi$) is chosen to match with the increase of $\hat{\psi}_1$.
For $\lvert\omega\rvert\in \left[\frac{1}{2},1\right]$ we have by \eqref{eq:OverlapPsi1_2} that
\begin{equation}\label{eq:overlapPsi1Phi}
  \lvert\hat{\psi}_1(\omega)\rvert^2+\lvert\varphi(\omega)\rvert^2
=
  \sin^2\left(\frac{\pi}{2}v(2\lvert\omega\rvert-1)\right)
  +
  \cos^2\left(\frac{\pi}{2}v(2\lvert\omega\rvert-1)\right)
=
  1.
\end{equation}

\begin{remark}
  Observe that in our setting it would not be useful to define the scaling function as a simple tensor product, namely
  \begin{align}\notag
    \hat{\Phi}(\omega)
  :=&
    \varphi(\omega_1)
    \varphi(\omega_2)
  \\\label{eq:Phi}
  =&
    \begin{cases}
      1                                                        & \text{for } \lvert\omega_1\rvert\leq\frac{1}{2},\lvert\omega_2\rvert\leq\frac{1}{2},        \\
      \cos\left(\frac{\pi}{2}v(2\lvert\omega_1\rvert-1)\right) & \text{for } \frac{1}{2}<\lvert\omega_1\rvert < 1, \lvert\omega_2\rvert\leq\frac{1}{2},      \\ 
      \cos\left(\frac{\pi}{2}v(2\lvert\omega_2\rvert-1)\right) & \text{for } \frac{1}{2}<\lvert\omega_2\rvert < 1, \lvert\omega_1\rvert\leq\frac{1}{2},      \\
      \cos\left(\frac{\pi}{2}v(2\lvert\omega_1\rvert-1)\right) 
      \cos\left(\frac{\pi}{2}v(2\lvert\omega_2\rvert-1)\right) & \text{for } \frac{1}{2}<\lvert\omega_1\rvert\leq 1, \frac{1}{2}<\lvert\omega_2\rvert\leq 1, \\
      0                                                        & \text{otherwise.}
    \end{cases}
  \end{align}
  Figure~\ref{fig:differentScalingFunctions} shows the two different scaling functions.
  Obviously, the first scaling function in Figure~\ref{fig:ConeScaling} aligns much better with the cones than the second in Figure~\ref{fig:TensorScaling}.
  \begin{figure}[htbp]
    \centering
    \subfloat[$\hat{\phi}(\omega)$.]{\label{fig:ConeScaling}
    \imageWithBorder{0.4\textwidth}{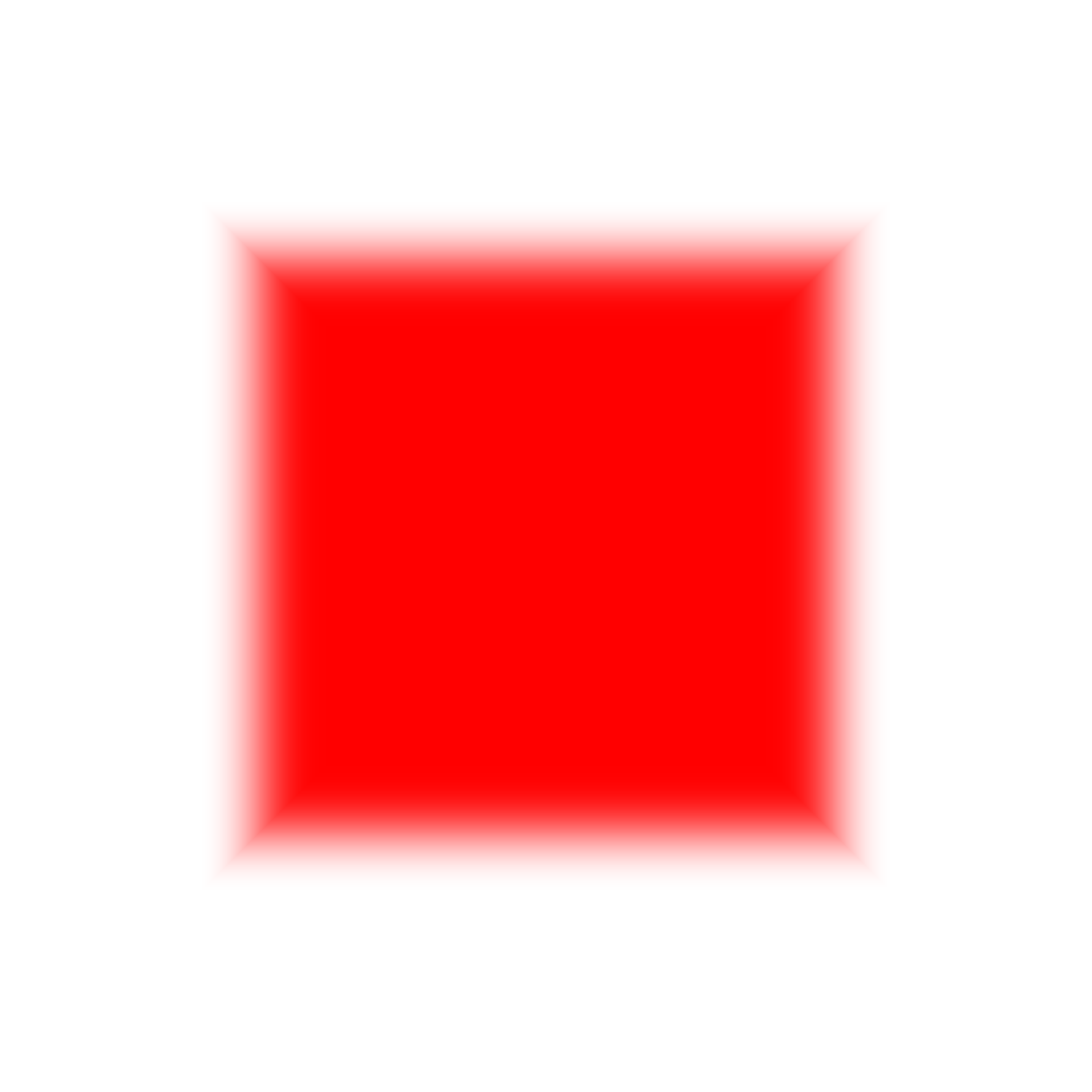}}
    \hspace{0.2cm}
    \subfloat[$\hat{\Phi}(\omega)$.]{\label{fig:TensorScaling}
    \imageWithBorder{0.4\textwidth}{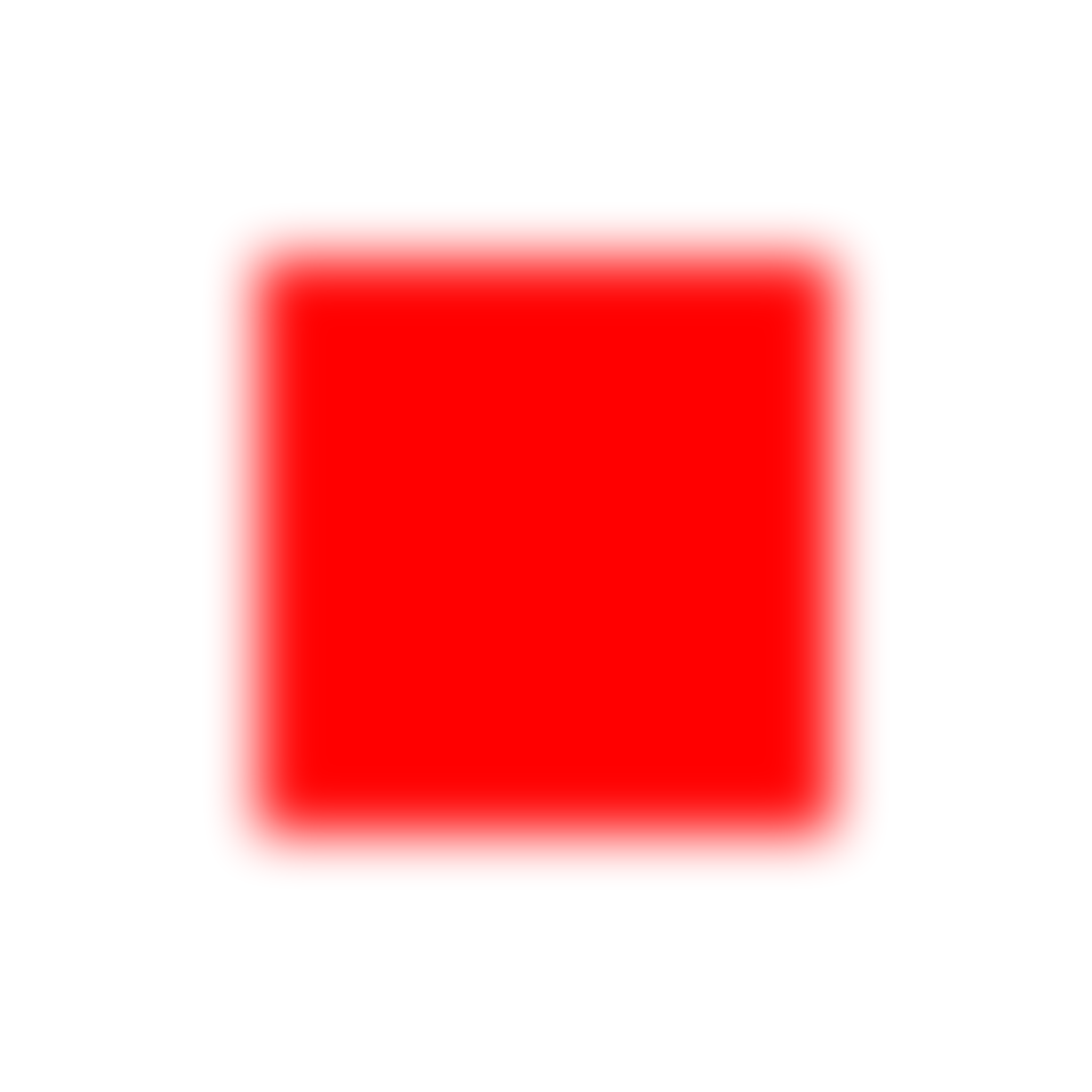}}
    \caption{The different scaling functions in \eqref{eq:phi} and \eqref{eq:Phi}.}
    \label{fig:differentScalingFunctions}
  \end{figure}
  In \cite{GL13} a new shearlet construction was introduced which is based on the scaling function in \eqref{eq:Phi}. 
  We discuss the new construction in Section~\ref{sec:smooth_shearlets}.
\end{remark}

\begin{remark}
  On the other hand it is possible to rewrite the definition of the original $\hat{\phi}$ as a shearlet-like tensor product. 
  We obtain a horizontal scaling function $\hat{\phi}^h$ and a vertical scaling function $\hat{\phi}^v$ as follows
  \begin{equation*}
    \hat{\phi}^h(\omega_1,\omega_2) 
  := 
    \varphi(\omega_1)\varphi\left(\frac{\omega_2}{2\omega_1}\right)
  \quad\text{and}\quad
    \hat{\phi}^v(\omega_1,\omega_2) 
  := 
    \varphi(\omega_2)\varphi\left(\frac{\omega_1}{2\omega_2}\right),
  \end{equation*}
  where
  \begin{equation*}
    \varphi\left(\frac{\omega_2}{2\omega_1}\right)
  =
    \begin{cases}
    1                                                                                              & \text{for }\lvert\omega_2\rvert\leq \lvert\omega_1\rvert,                       \\
    \cos\left(\frac{\pi}{2}v\left(\left\lvert\frac{\omega_2}{\omega_1}\right\rvert-1\right)\right) & \text{for }\lvert\omega_1\rvert < \lvert\omega_2\rvert < 2\lvert\omega_1\rvert, \\
    0                                                                                              & \text{otherwise.}
    \end{cases}
  \end{equation*}
  Thus, $\varphi\left(\frac{\omega_2}{2\omega_1}\right)$ is a continuous extension of the characteristic function of the horizontal cone $\cone^h$.
\end{remark}

We set
\begin{equation*}
  \phi_{a,s,t}(x) 
= 
  \phi_{t}(x) 
= 
  \phi(x-t).
\end{equation*}
Note that there is neither dilation nor shear for the scaling function, only translation. 
Consequently,  the index ``$a,s,t$'' from  the shearlet $\psi$  reduces to ``$t$''. 
We further obtain
\begin{equation*}
  \hat{\phi_t}(\omega) 
= 
  \e^{-2\pi i \langle \omega,t \rangle} \hat{\phi}(\omega).
\end{equation*}
The transform can be obtained similar as before, namely
\begin{equation*}
  \cSH _\phi(f)(a,s,t)
=
  \langle f, \phi_{t} \rangle.
\end{equation*}

\section{Fast Computation of the Finite Discrete Shearlet Transform} \label{sec:computation_discrete_shearlet_transform}
We consider digital images in $\RR^{M\times N}$ as functions sampled on the grid
$\left\{\left(\frac{m_1}{M},\frac{m_2}{N}\right) : (m_1,m_2)\in\cG  \right\}$
with
$\cG := \{(m_1,m_2) : m_1 = 0,\ldots, M-1,\ m_2=0,\ldots N-1\}$
and assume periodic continuation over the boundary.

The discrete shearlet transform is basically known, but in contrast to the existing literature we present here a fully discrete setting. 
That is, we do not only discretize the involved parameters $a$, $s$ and $t$ but also consider only a finite number of discrete translations $t$. 
Additionally, our setting discretizes the translation parameter $t$ on a rectangular grid and independent of the dilation and shear parameter. 
See Section~\ref{section:Remarks} for further remarks on this topic.

\subsection{Finite Discrete Shearlets}
Let $j_0:=\lfloor \frac{1}{2}\log_2 \max\{M,N\}\rfloor$ be the number of considered scales.
To obtain a discrete shearlet transform, we
discretize the dilation, shear and translation parameters as
\begin{align}\notag
  a_j
&:=
  2^{-2j} = \frac{1}{4^j},\quad j= 0,\ldots,j_0-1,
  \\\notag
  s_{j,k}
&:=
  k2^{-j},\quad -2^j\leq k\leq 2^j,
  \\\label{eq:discrete_t}
  t_m
&:=
  \left(\frac{m_1}{M},\frac{m_2}{N}\right),\quad m \in \cG.
\end{align}
With these notations our shearlet becomes\index{shearlet!discrete}
$
  \psi_{j,k,m}(x)
:=
  \psi_{a_j,s_{j,k},t_{m}} (x)
=
  \psi(A_{a_j,\frac{1}{2}}^{-1}S_{s_{j,k}}^{-1}(x-t_{m}))
$.
Observe that compared to the continuous shearlets defined in \eqref{eq:psi_ast} we omit the factor $a^{-\frac{3}{4}}$.
In Fourier domain we obtain
\begin{equation*}
 \hat{\psi}_{j,k,m}(\omega)
=
  \hat{\psi} (A_{a_j}^\tT S_{s_{j,k}}^\tT\omega)
  \e^{-2\pi i \langle \omega, t_m \rangle}
=
  \hat{\psi}_1
  \left(
  4^{-j}\omega_1
  \right)
  \hat{\psi}_2
  \left(
  2^j\frac{\omega_2}{\omega_1} + k
  \right)
  \e^{-2\pi i \langle \omega, \binom{m_1/M}{m_2/N} \rangle},
\quad 
  \omega \in \Omega,
\end{equation*}
where
$
  \Omega
:=
  \left\{ 
    (\omega_1,\omega_2)
  :
    \omega_1 
  =
    -\left\lfloor\frac{M}{2}\right\rfloor, \ldots , \left\lceil\frac{M}{2}\right\rceil -1,
    \omega_2 
  = 
    -\left\lfloor\frac{N}{2}\right\rfloor, \ldots , \left\lceil\frac{N}{2}\right\rceil -1
  \right\}.
$
The chosen discretization of the dilation and shear parameter together with the support properties of $\psi$ induces the frequency tiling shown in Figure~\ref{fig:frequency_tiling_indices}.

We consider these shearlets on the previously introduced cones.
By definition the parameters fulfill $a\leq 1$ and $\lvert s\rvert\leq 1$.
Therefore we see that a cut off due to the cone boundaries happens only for $\lvert k \rvert = 2^j$ where $\lvert s\rvert = 1$.
For both cones we have for $\lvert s\rvert=1$ two ``half'' shearlets with a gap at the seam line.
None of the shearlets is defined on the seam line $\cone^\times$. 
To obtain full shearlets at the seam lines we ``glue'' the three parts together,
that is, we define for $\lvert k \rvert = 2^j$ a sum of shearlets
\begin{equation*}
  \hat{\psi}^{h\times v}_{j,k,m}
:=
  \hat{\psi}^h_{j,k,m}
  +
  \hat{\psi}^v_{j,k,m}
  +
  \hat{\psi}^\times_{j,k,m}.
\end{equation*}
We define the \emph{discrete shearlet transform}\index{shearlet transform!discrete} as
\begin{equation*}
  \cSH (f)(\kappa,j,k,m)
:=
  \begin{cases}
    \langle f,\phi_m \rangle                   & \text{for }\kappa = 0,                             \\
    \langle f,\psi^\kappa_{j,k,m} \rangle      & \text{for }\kappa \in \{h,v\},                     \\
    \langle f,\psi^{h\times v}_{j,k,m} \rangle & \text{for }\kappa = \times, \lvert k \rvert =  2^j
  \end{cases}
\end{equation*}
where $j= 0,\ldots,j_0-1$, $-2^j+1\leq k \leq 2^j-1$ and $m\in \cG $ if not stated otherwise.
The shearlet transform can be efficiently realized by applying the \fftn\ and its inverse \ifftn.

Using Parseval's formula the discrete shearlet transform is computed for $\kappa=h$ as follows 
(observe that $\hat{\psi}$ is real):
\begin{align}\notag
  \cSH (f)(h,j,k,m)  \notag
&=
  \langle f, \psi_{j,k,m}^h \rangle
=
  \frac{1}{MN}
  \langle \hat{f}, \hat{\psi}_{j,k,m}^h \rangle
  \\\notag
&=
  \frac{1}{MN}
  \sum_{\omega\in\Omega}
  \overline{
  \e^{-2\pi i\langle \omega,\binom{m_1/M}{m_2/N} \rangle}
  \hat{\psi}
  (4^{-j}\omega_1,4^{-j}k\omega_1+2^{-j}\omega_2)
  }
  \hat{f}(\omega_1,\omega_2)
  \\\notag
&=
  \frac{1}{MN}
  \sum_{\omega\in\Omega}
  \hat{\psi}
  (4^{-j}\omega_1,4^{-j}k\omega_1+2^{-j}\omega_2)
  \hat{f}(\omega_1,\omega_2)
  \e^{2\pi i\left\langle \omega,\binom{m_1/M}{m_2/N} \right\rangle}.
\end{align}
With $\hat{g}_{j,k}(\omega) := \hat{\psi}(4^{-j}\omega_1,4^{-j}k\omega_1+2^{-j}\omega_2)\hat{f}(\omega_1,\omega_2)$ this becomes
\begin{align}\notag
  \cSH (f)(h,j,k,m)
&=
  \frac{1}{MN}
  \sum_{\omega\in\Omega}
  \hat{g}_{j,k}(\omega)
  \e^{2\pi i\left\langle \omega,\binom{m_1/M}{m_2/N} \right\rangle}.
\end{align}
Since $\hat{g}_{j,k}(\omega)\in\CC^{M\times N}$ final step in computation of the shearlet transform is an inverse FFT of $\hat{g}_{j,k}$, thus
\begin{align}\notag
  \cSH (f)(h,j,k,m)
&=
  \ifftn(\hat{g}_{j,k})
  \\\label{eq:discreteHorizontalTrans}
&=
  \ifftn(\hat{\psi}
  (4^{-j}\omega_1,4^{-j}k\omega_1+2^{-j}\omega_2)
  \hat{f}(\omega_1,\omega_2)).
\end{align}

For the vertical cone, i.e., $\kappa=v$, the transform reads
\begin{equation} \label{eq:discreteVerticalTrans}
  \cSH (f)(v,j,k,m)
=
  \ifftn(\hat{\psi}
  (4^{-j}\omega_2,4^{-j}k\omega_2+2^{-j}\omega_1)
  \hat{f}(\omega_1,\omega_2))
\end{equation}
and for the seam line part with $\lvert k \rvert=2^j$ we use the ``glued'' shearlets leading to
\begin{equation} \label{eq:discreteSeamTrans}
  \cSH (f)_{\psi^{h\times v}}(j,k,m)
=
  \ifftn(\hat{\psi}^{h\times v}
  (4^{-j}\omega_1,4^{-j}k\omega_1+2^{-j}\omega_2)
  \hat{f}(\omega_1,\omega_2)).
\end{equation}
Finally for the low-pass with
$\hat{g}_0(\omega_1,\omega_2):=\hat{\phi}(\omega_1,\omega_2)\hat{f}(\omega_1,\omega_2)$ 
and similar steps as above the transform is computed as
\begin{align}\notag
  \cSH _\phi(f)(m)
&=
  \langle f, \phi_{m} \rangle
  \\\notag
&=
  \frac{1}{MN}
  \langle \hat{f}, \hat{\phi}_{m} \rangle
  \\\notag
&=
  \frac{1}{MN}
  \sum_{\omega\in\Omega}
  \overline{
  \e^{-2\pi i\langle \omega,\binom{m_1/M}{m_2/N} \rangle}
  \hat{\phi}
  (\omega_1,\omega_2)
  }
  \hat{f}(\omega_1,\omega_2)
  \\\notag
&=
  \frac{1}{MN}
  \sum_{\omega\in\Omega}
  \e^{+2\pi i\langle \omega,\binom{m_1/M}{m_2/N}  \rangle}
  \hat{\phi}
  (\omega_1,\omega_2)
  \hat{f}(\omega_1,\omega_2)
  \\\notag
&=
  \frac{1}{MN}
  \sum_{\omega\in\Omega}
  \e^{+2\pi i\langle \omega,\binom{m_1/M}{m_2/N}  \rangle}
  \hat{g}_0(\omega)
  \\\notag
&=
  \ifftn(\hat{g}_0)
  \\\label{eq:discreteLowTrans}
&=
  \ifftn(\hat{\phi}
  (\omega_1,\omega_2)\hat{f}(\omega_1,\omega_2)).
\end{align}

The complete shearlet transform is the combination of \eqref{eq:discreteHorizontalTrans} to \eqref{eq:discreteLowTrans}. We summarize
\begin{small}\begin{equation}\label{eq:shearletTransform}
  \cSH (f)(\kappa,j,k,m)
  =
  \begin{cases}
    \ifftn(\hat{\phi}(\omega_1,\omega_2)\hat{f}(\omega_1,\omega_2))
    &\text{for }\kappa = 0,
    \\
    \ifftn(\hat{\psi}
    (4^{-j}\omega_1,4^{-j}k\omega_1+2^{-j}\omega_2)
    \hat{f}(\omega_1,\omega_2))
    &\text{for }\kappa = h, \lvert k \rvert\leq 2^j-1,
    \\
    \ifftn(\hat{\psi}
    (4^{-j}\omega_2,4^{-j}k\omega_2+2^{-j}\omega_1)
    \hat{f}(\omega_1,\omega_2))
    &\text{for }\kappa = v, \lvert k \rvert\leq 2^j-1,
    \\
    \ifftn(\hat{\psi}^{h\times v}
    (4^{-j}\omega_1,4^{-j}k\omega_1+2^{-j}\omega_2)
    \hat{f}(\omega_1,\omega_2))
    &\text{for }\kappa \neq 0, \lvert k \rvert = 2^j.
  \end{cases}
\end{equation}\end{small}

\subsection{A Discrete Shearlet Frame}
In view of the inverse shearlet transform we prove that our discrete shearlets constitute a Parseval frame of the finite Euclidean space $L_2(\cG)$.
Recall that for a Hilbert space $\cH$ a sequence
$\{u_i : i\in\cI \}$ is a \emph{frame}\index{frame} 
if and only if there exist constants $0 < A\leq B <\infty$ such that
\begin{equation*}
  A\lVert f\rVert_\cH^2
\leq
  \sum_{i\in\cI}
  \lvert\langle
    f, u_i
  \rangle\rvert^2
\leq
  B \lVert f\rVert_\cH^2
\quad\text{for all }
  f\in\cH .
\end{equation*}
The frame is called \emph{tight}\index{frame!tight} if $A=B$ and a \emph{Parseval}\index{frame!Parseval} frame if $A = B = 1$.
Thus, for Parseval frames we have that
\begin{equation*}
  \lVert f\rVert_\cH^2
=
  \sum_{i\in\cI}
  \lvert\langle f,u_i \rangle\rvert^2
\quad\text{for all }
  f\in\cH 
\end{equation*}
which is equivalent to the reconstruction formula
\begin{equation*} 
  f
=
  \sum_{i\in\cI}
  \langle f,u_j\rangle u_j
\quad\text{for all }
  f\in\cH .
\end{equation*}
Further details on frames can be found in \cite{Chr03} and \cite{Mal08}. 
In the $d$-dimensional Euclidean space we can
arrange the frame elements $u_i$, $i = 1,\ldots,\widetilde{d} \geq d$ as rows of a matrix $U$.
Then we have indeed a frame if $U$ has full rank and a Parseval frame if and only if $U^\tT U = I_d$.
Note that $U U^\tT = I_{\widetilde{d}}$ is only true if the frame is an orthonormal basis.
The Parseval frame transform and its inverse read
\begin{equation}\label{eq:shearlet_transform_as_matrix}
  (\langle f,u_i \rangle)_{i =1}^{\widetilde{d}} 
= 
  U f
\quad \text{and} \quad
  f 
= 
  U^\tT (\langle f,u_i \rangle)_{i =1}^{\widetilde{d}}.
\end{equation}
By the following theorem our shearlets provide such a convenient system.

\begin{theorem} \index{frame!discrete shearlet}
  The discrete shearlet system
  \begin{align*}
  &\{
    \psi^h_{j,k,m}(\omega) : j= 0,\ldots,j_0-1, -2^{j}+1\leq k\leq 2^j-1, m\in\cG 
    \}\\
  &\cup
    \{
    \psi^v_{j,k,m}(\omega) : j= 0,\ldots,j_0-1, -2^{j}+1\leq k\leq 2^j-1, m\in\cG 
    \}\\
  &\cup
    \{
    \psi^{h\times v}_{j,k,m}(\omega) : j= 0,\ldots,j_0-1, \lvert k \rvert = 2^j, m\in\cG 
    \}\\
   &\cup
    \{
    \phi_{m}(\omega) : m\in\cG
    \}
  \end{align*}
  provides a Parseval frame for $L^2(\cG)$.
\end{theorem}
\begin{proof}
We have to show that
\begin{align*}
  \lVert f\rVert_{L^2(\cG)}^2
&=
  \sum_{\kappa\in\{h,v\}} 
  \sum_{j=0}^{j_0-1} 
  \sum_{k=-2^j+1}^{2^j-1} 
  \sum_{m\in \cG}
    \lvert\langle f,\psi_{j,k,m}^\kappa \rangle\rvert^2
+
  \sum_{j=0}^{j_0-1}
  \sum_{k=\pm 2^j} 
  \sum_{m\in \cG}
    \lvert\langle f,\psi_{j,k,m}^{h\times  v} \rangle\rvert^2
+
  \sum_{m\in \cG}
  \lvert\langle f,\phi_{m} \rangle\rvert^2 
\\
&=: 
  C.
\end{align*}
Since 
$\lVert f\rVert_{L^2(\cG)} = \lVert f\rVert_F^2 = \frac{1}{MN}\lVert\hat{f}\rVert_F^2$ 
(Parseval's formula)
it is sufficient to show that $C$ is equal to $\frac{1}{MN}\lVert\hat{f}\rVert_F^2$.

By \eqref{eq:discreteHorizontalTrans} we know that
\begin{align*}
  \langle f,\psi_{j,k,m}^h \rangle
&=
  \frac{1}{MN}
  \sum_{\omega\in\Omega}
  \e^{2\pi i\langle \omega,\binom{m_1/M}{m_2/N} \rangle}
  \hat{g}_{j,k}(\omega)
=
  g_{j,k}(m).
\end{align*}
We further obtain
\begin{align*}
  \sum_{m\in\cG}
  \lvert\langle
    f, \psi_{j,k,m}^h
  \rangle\rvert^2
&=
  \sum_{m\in\cG}
  \lvert g_{j,k}(m)\rvert^2
=
  \lVert g_{j,k}\rVert_F^2.
\end{align*}
Consequently, with Parseval's formula
\begin{align*}
  \lVert g_{j,k}\rVert_F^2
&=
  \frac{1}{MN}
  \lVert\hat{g}_{j,k}\rVert_F^2
  =
  \frac{1}{MN}
  \sum_{\omega\in\Omega}
  \lvert \hat{g}_{j,k}(\omega)\rvert^2
  \\
&=
  \frac{1}{MN}
  \sum_{\omega\in\Omega}
  \lvert
  \hat{\psi}
  (4^{-j}\omega_1,4^{-j}k\omega_1+2^{-j}\omega_2)\hat{f}(\omega_1,\omega_2)
  \rvert^2
  \\
&=
  \frac{1}{MN}
  \sum_{\omega\in\Omega}
  \lvert
  \hat{\psi}
  (4^{-j}\omega_1,4^{-j}k\omega_1+2^{-j}\omega_2)
  \rvert^2
  \lvert
  \hat{f}(\omega_1,\omega_2)
  \rvert^2.
\end{align*}
Analogously we obtain for the vertical part
\begin{equation*}
  \sum_{m\in\cG}
  \lvert\langle
    f,\psi_{j,k,m}^v
  \rangle\rvert^2
  =
  \frac{1}{MN}
  \sum_{\omega\in\Omega}
  \lvert
    \hat{\psi}
    (4^{-j}\omega_2,4^{-j}k\omega_2+2^{-j}\omega_1)
  \rvert^2
  \lvert
    \hat{f}(\omega_1,\omega_2)
  \rvert^2.
\end{equation*}
Using these results we can conclude for the seam-line part
\begin{align*}
  \sum_{m\in\cG}
  \lvert\langle f,\psi_{j,k,m}^{h\times  v} \rangle\rvert^2
&=
  \frac{1}{MN}
  \sum_{\omega\in\Omega}
  \lvert
    \hat{\psi}
    (4^{-j}\omega_1,4^{-j}k\omega_1+2^{-j}\omega_2)
  \rvert^2
  \lvert
    \hat{f}(\omega_1,\omega_2)
  \rvert^2
  \chi_{\cone^h}
\\
&\phantom{-}+
  \frac{1}{MN}
  \sum_{\omega\in\Omega}
  \lvert
    \hat{\psi}
    (4^{-j}\omega_2,4^{-j}k\omega_2+2^{-j}\omega_1)
  \rvert^2
  \lvert
    \hat{f}(\omega_1,\omega_2)
  \rvert^2
  \chi_{\cone^v}
\\
&\phantom{-}+
  \frac{1}{MN}
  \sum_{\omega\in\Omega}
  \lvert
    \hat{\psi}
    (4^{-j}\omega_1,4^{-j}k\omega_1+2^{-j}\omega_2)
  \rvert^2
  \lvert
    \hat{f}(\omega_1,\omega_2)
  \rvert^2
  \chi_{\cone^\times}.
\end{align*}
For the remaining low-pass part we get similarly
\begin{align*}
  \sum_{m\in\cG}
  \lvert\langle
  f, \phi_{m}
  \rangle\rvert^2
&=
  \frac{1}{MN}
  \sum_{m\in\cG}
  \lvert\langle
  \hat{f}, \hat{\phi}_{m}
  \rangle\rvert^2
  \\
&=
  \frac{1}{MN}
  \sum_{m\in\cG}
  \Bigl\lvert
  \sum_{\omega\in\Omega}
  \overline{
  \hat{\phi}_{m}(\omega)
  }
  \hat{f}(\omega)
  \Bigr\rvert^2
  \\
&=
  \frac{1}{MN}
  \sum_{m\in\cG}
  \Bigl\lvert
  \sum_{\omega\in\Omega}
  \e^{2\pi i\langle \omega,\binom{m_1/M}{m_2/N} \rangle}
  \hat{\phi}
  (\omega_1,\omega_2)\hat{f}(\omega_1,\omega_2)
  \Bigr\rvert^2
  \\
\intertext{with $\hat{g}_0(\omega) := \hat{\phi}(\omega_1,\omega_2)\hat{f}(\omega_1,\omega_2)$}
  \sum_{m\in\cG}
  \lvert\langle
  f, \phi_{m}
  \rangle\rvert^2
&=
  \sum_{m\in\cG}
  \Bigl\lvert
  \frac{1}{MN}
  \sum_{\omega\in\Omega}
  \e^{2\pi i\langle \omega,\binom{m_1/M}{m_2/N} \rangle}
  \hat{g}_0(\omega)
  \Bigr\rvert^2
  \\
&=
  \sum_{m\in\cG}
  \lvert
  g_0(m)
  \rvert^2
  =
  \lVert g_0\rVert_F^2
  =
  \frac{1}{MN}
  \lVert\hat{g}_0\rVert_F^2
  \\
&=
  \frac{1}{MN}
  \sum_{\omega\in\Omega}
  \lvert
  \hat{\phi}(\omega_1,\omega_2)\hat{f}(\omega_1,\omega_2)
  \rvert^2
  \\
&=
  \frac{1}{MN}
  \sum_{\omega\in\Omega}
  \lvert
  \hat{\phi}(\omega_1,\omega_2)
  \rvert^2
  \lvert
  \hat{f}(\omega_1,\omega_2)
  \rvert^2.
\end{align*}
Let us put the pieces together:
\begin{align*}
  C
&=
  \sum_{\kappa\in\{h,v\}}\sum_{j= 0}^{j_0-1}\sum_{k=-2^j+1}^{2^j-1}\sum_{m\in\cG}
  \lvert\langle f,\psi_{j,k,m}^\kappa \rangle\rvert^2
  +
  \sum_{j= 0}^{j_0-1}\sum_{k=\pm 2^j}\sum_{m\in\cG}
  \lvert\langle f,\psi_{j,k,m}^{h\times  v} \rangle\rvert^2
  +
  \sum_{m\in\cG}
  \lvert\langle f,\phi_{m} \rangle\rvert^2
\\
&=
  \sum_{j= 0}^{j_0-1}
  \sum_{k=-2^j+1}^{2^j-1}
  \frac{1}{MN}
  \sum_{\omega\in\Omega}
  \lvert
  \hat{\psi}
  (4^{-j}\omega_1,4^{-j}k\omega_1+2^{-j}\omega_2)
  \rvert^2
  \lvert
  \hat{f}(\omega_1,\omega_2)
  \rvert^2
\\
&\quad+
  \sum_{j= 0}^{j_0-1}
  \sum_{k=-2^j+1}^{2^j-1}
  \frac{1}{MN}
  \sum_{\omega\in\Omega}
  \lvert
  \hat{\psi}
  (4^{-j}\omega_2,4^{-j}k\omega_2+2^{-j}\omega_1)
  \rvert^2
  \lvert
  \hat{f}(\omega_1,\omega_2)
  \rvert^2
\\
&\quad+
  \sum_{j= 0}^{j_0-1}\sum_{k=\pm 2^j}
  \left(\frac{1}{MN}\right.
  \sum_{\omega\in\Omega}
  \lvert
  \hat{\psi}
  (4^{-j}\omega_1,4^{-j}k\omega_1+2^{-j}\omega_2)
  \rvert^2
  \lvert
  \hat{f}(\omega_1,\omega_2)
  \rvert^2
  \chi_{\cone^h}
\\
&\quad+
  \frac{1}{MN}
  \sum_{\omega\in\Omega}
  \lvert
  \hat{\psi}
  (4^{-j}\omega_2,4^{-j}k\omega_2+2^{-j}\omega_1)
  \rvert^2
  \lvert
  \hat{f}(\omega_1,\omega_2)
  \rvert^2
  \chi_{\cone^v}
\\
&\quad+
  \left.
  \frac{1}{MN}
  \sum_{\omega\in\Omega}
  \lvert
    \hat{\psi}
    (4^{-j}\omega_1,4^{-j}k\omega_1+2^{-j}\omega_2)
  \rvert^2
  \lvert
    \hat{f}(\omega_1,\omega_2)
  \rvert^2
  \chi_{\cone^\times}
  \right)
\\
&\quad+
  \frac{1}{MN}
  \sum_{\omega\in\Omega}
  \lvert
  \hat{\phi}(\omega_1,\omega_2)
  \rvert^2
  \lvert
  \hat{f}(\omega_1,\omega_2)
  \rvert^2.
\end{align*}
We can group the sums by the different sets and obtain
\begin{align*}
  C
&=
  \frac{1}{MN}
  \sum_{\omega\in\Omega}
  \sum_{j= 0}^{j_0-1}
  \sum_{k=-2^j}^{2^j}
  \lvert
  \hat{\psi}
  (4^{-j}\omega_1,4^{-j}k\omega_1+2^{-j}\omega_2)
  \rvert^2
  \lvert
  \hat{f}(\omega_1,\omega_2)
  \rvert^2
  \chi_{\cone^h}
\\
&\quad+
  \frac{1}{MN}
  \sum_{\omega\in\Omega}
  \sum_{j=0}^{j_0-1}
  \sum_{k=-2^j}^{2^j}
  \lvert
  \hat{\psi}
  (4^{-j}\omega_2,4^{-j}k\omega_2+2^{-j}\omega_1)
  \rvert^2
  \lvert
  \hat{f}(\omega_1,\omega_2)
  \rvert^2
  \chi_{\cone^v}
\\
&\quad+
  \frac{1}{MN}
  \sum_{\omega\in\Omega}
  \sum_{j= 0}^{j_0-1}
  \lvert
  \hat{f}(\omega_1,\omega_2)
  \rvert^2
  \lvert
  \underbrace{
  \hat{\psi}
  (4^{-j}\omega_1,0)
  }_{
  = 1
  }
  \rvert^2
  \chi_{\cone^\times}
  +
  \frac{1}{MN}
  \sum_{\omega\in\Omega}
  \lvert
  \hat{\phi}(\omega_1,\omega_2)
  \rvert^2
  \lvert
  \hat{f}(\omega_1,\omega_2)
  \rvert^2.
\end{align*}
Using the definition of $\hat{\psi}$ in \eqref{eq:horizontalShearlet} (or \eqref{eq:verticalShearlet} and \eqref{eq:seamlineShearlet}, respectively), we can conclude
\begin{align*}
  C
&=
  \frac{1}{MN}
  \sum_{\omega\in\cone^h}
  \lvert
  \hat{f}(\omega_1,\omega_2)
  \rvert^2
  \underbrace{
  \sum_{j= 0}^{j_0-1}
  \lvert
    \hat{\psi}_1(4^{-j}\omega_1)
  \rvert^2
  }_{
  \substack{
    = 1\text{ for }\lvert\omega_1\rvert\geq 1
    \\
    \text{ (see Theorem~\ref{thm:PropertyPsi1})}
  }
  }
  \underbrace{
  \sum_{k=-2^j}^{2^j}
  \lvert
    \hat{\psi}_2(2^j\frac{\omega_2}{\omega_1} + k)
  \rvert^2
  }_{
  = 1 \text{ (see Theorem~\ref{thm:PropertyPsi2})}
  }
\\
&\quad+
  \frac{1}{MN}
  \sum_{\omega\in\cone^v}
  \lvert
  \hat{f}(\omega_1,\omega_2)
  \rvert^2
  \underbrace{
  \sum_{j=0}^{j_0-1}
  \lvert
    \hat{\psi}_1(4^{-j}\omega_2)
  \rvert^2
  }_{
  = 1\text{ for }\lvert\omega_2\rvert\geq 1
  }
  \underbrace{
  \sum_{k=-2^j}^{2^j}
  \lvert
    \hat{\psi}_2(2^j\frac{\omega_1}{\omega_2} + k)
  \rvert^2
  }_{
  = 1
  }
\\
&\quad+
  \frac{1}{MN}
  \sum_{\omega\in\cone^\times}
  \lvert
  \hat{f}(\omega_1,\omega_2)
  \rvert^2
  +
  \frac{1}{MN}
  \sum_{\omega\in\Omega}
  \quad
  \lvert
  \underbrace{
    \hat{\phi}(\omega_1,\omega_2)
  }_{
    \makebox[0mm][c]{ 
    \begin{scriptsize}$= 1$ for $\omega\in[-\frac{1}{2},\frac{1}{2}]^2$\end{scriptsize}}
  }
  \rvert^2
  \lvert
  \hat{f}(\omega_1,\omega_2)
  \rvert^2.
\end{align*}
With the properties of $\hat{\psi}_1$ and $\hat{\psi}_2$ 
(see Theorems~\ref{thm:PropertyPsi1} and \ref{thm:PropertyPsi2}) 
we obtain two sums, 
one for the overlapping domain $\cone^\square$ (see \eqref{eq:overlappingDomain}) 
and one for the remaining part
\begin{align*}
  C
&=
  \frac{1}{MN}
  \sum_{\omega\in\Omega\setminus\cone^\square}
  \lvert
    \hat{f}(\omega_1,\omega_2)
  \rvert^2
\\
&\phantom{=\ }
  +
  \sum_{\omega\in\cone^\square}
  \lvert
  \hat{f}(\omega_1,\omega_2)
  \rvert^2
  \left(
    \sum_{j=0}^{j_0-1}
    \lvert
    \hat{\psi}_1(4^{-j}\omega_1)
    \rvert^2
    +
    \sum_{j=0}^{j_0-1}
    \lvert
    \hat{\psi}_1(4^{-j}\omega_2)
    \rvert^2
    +
    \lvert
    \hat{\phi}(\omega_1,\omega_2)
    \rvert^2
  \right)
\end{align*}
where we can split up the second sum as
\begin{align*}
  C
&=
  \frac{1}{MN}
  \sum_{\omega\in\Omega\setminus\cone^\square}
  \lvert
  \hat{f}(\omega_1,\omega_2)
  \rvert^2
\\
&\phantom{=\ }
  +
  \frac{1}{MN}
  \sum_{\omega\in\cone^h\cap\cone^\square}
  \lvert
  \hat{f}(\omega_1,\omega_2)
  \rvert^2
  \sin^2\left(\frac{\pi}{2}v(2\lvert\omega_1\rvert-1)\right)
\\
&\phantom{=\ }
  +
  \frac{1}{MN}
  \sum_{\omega\in\cone^v\cap\cone^\square}
  \lvert
  \hat{f}(\omega_1,\omega_2)
  \rvert^2
  \sin^2\left(\frac{\pi}{2}v(2\lvert\omega_2\rvert-1)\right)
\\
&\phantom{=\ }
  +
  \frac{1}{MN}
  \sum_{\omega\in\cone^h\cap\cone^\square}
  \lvert
  \hat{f}(\omega_1,\omega_2)
  \rvert^2
  \cos^2\left(\frac{\pi}{2}v(2\lvert\omega_1\rvert-1)\right)
\\
&\phantom{=\ }
  +
  \frac{1}{MN}
  \sum_{\omega\in\cone^v\cap\cone^\square}
  \lvert
  \hat{f}(\omega_1,\omega_2)
  \rvert^2
  \cos^2\left(\frac{\pi}{2}v(2\lvert\omega_2\rvert-1)\right).
\end{align*}
With the overlap (see \eqref{eq:overlapPsi1Phi}) we can continue
\begin{align*}
  C
&=
  \frac{1}{MN}
  \sum_{\omega\in\Omega\setminus\cone^\square}
  \lvert
  \hat{f}(\omega_1,\omega_2)
  \rvert^2
\\
&\quad+
  \frac{1}{MN}
  \sum_{\omega\in\cone^h\cap\cone^\square}
  \lvert
  \hat{f}(\omega_1,\omega_2)
  \rvert^2
  \underbrace{
  \left(
    \sin^2\left(\frac{\pi}{2}v(2\lvert\omega_1\rvert-1)\right) + \cos^2\left(\frac{\pi}{2}v(2\lvert\omega_1\rvert-1)\right)
  \right)
  }_{
  = 1\text{ (see \eqref{eq:overlapPsi1Phi})}
  }
\\
&\quad+
  \frac{1}{MN}
  \sum_{\omega\in\cone^v\cap\cone^\square}
  \lvert
  \hat{f}(\omega_1,\omega_2)
  \rvert^2
  \underbrace{
  \left(
    \sin^2\left(\frac{\pi}{2}v(2\lvert\omega_2\rvert-1)\right) + \cos^2\left(\frac{\pi}{2}v(2\lvert\omega_2\rvert-1)\right)
  \right)
  }_{
  = 1\text{ (see \eqref{eq:overlapPsi1Phi})}
  }.
\end{align*}
Finally, we obtain
\begin{equation*}
  C
=
  \frac{1}{MN}
  \sum_{\omega\in\Omega}
  \lvert
  \hat{f}(\omega_1,\omega_2)
  \rvert^2
=
  \frac{1}{MN}
  \lVert\hat{f}\rVert_F^2
=
  \lVert f\rVert_F^2.
  \qedhere
\end{equation*}
\end{proof}

\subsection{Inversion of the Shearlet Transform}
Having the discrete Parseval frame the inversion of the shearlet transform is straightforward: 
multiply each coefficient with the respective shearlet and sum over all involved parameters. 
As inversion formula we obtain
\begin{equation*}\index{shearlet transform!inverse}
  f
=
  \sum_{\kappa\in\{h,v\}}\sum_{j= 0}^{j_0-1}\sum_{k=-2^j+1}^{2^j-1}\sum_{m\in\cG}
  \langle f,\psi_{j,k,m}^\kappa \rangle \psi_{j,k,m}^\kappa
+
  \sum_{j= 0}^{j_0-1}\sum_{k=\pm 2^j}\sum_{m\in\cG}
  \langle f,\psi_{j,k,m}^{h\times  v} \rangle \psi_{j,k,m}^{h\times  v}
+
  \sum_{m\in\cG}
  \langle f,\phi_{m} \rangle \phi_{m}.
\end{equation*}
The actual computation of $f$ from given coefficients $c(\kappa,j,k,m)$ is done in Fourier domain. 
Due to the linearity of the Fourier transform this is
\begin{equation*}
  \hat{f}
=
  \sum_{\kappa\in\{h,v\}}\sum_{j= 0}^{j_0-1}\sum_{k=-2^j+1}^{2^j-1}\sum_{m\in\cG}
  \langle f,\psi_{j,k,m}^\kappa \rangle \hat{\psi}_{j,k,m}^\kappa
+
  \sum_{j=0}^{j_0-1}\sum_{k=\pm 2^j}\sum_{m\in\cG}
  \langle f,\psi_{j,k,m}^{h\times  v} \rangle \hat{\psi}_{j,k,m}^{h\times  v}
+
  \sum_{m\in\cG}
  \langle f,\phi_{m} \rangle \hat{\phi}_{m}.
\end{equation*}
We take a closer look at the part for the horizontal cone where we have
\begin{align*}
  \hat{f}(\omega)\chi_{\cone^h}
&=
  \sum_{j=0}^{j_0-1}
  \sum_{k=-2^j+1}^{2^j-1}
  \sum_{m\in\cG}
  \langle
  f,
  \psi_{j,k,m}
  \rangle
  \hat{\psi}^h_{j,k,m}(\omega)
  \\
&=
  \sum_{j=0}^{j_0-1}
  \sum_{k=-2^j+1}^{2^j-1}
  \sum_{m\in\cG}
  c(h,j,k,m)
  \e^{-2\pi i \langle \omega,\binom{m_1/M}{m_2/N} \rangle}
  \hat{\psi}
  (4^{-j}\omega_1,4^jk\omega_1+2^{-j}\omega_2).
\end{align*}
The inner sum can be interpreted as a two-dimensional discrete Fourier transform and is computed with a FFT and thus we may write
\begin{equation*}\label{eq:inverseHorizontal}
  \hat{f}(\omega)\chi_{\cone^h}
=
  \sum_{j=0}^{j_0-1}
  \sum_{k=-2^j+1}^{2^j-1}
  \fftn(c(h,j,k,\,\cdot\,))(\omega_1,\omega_2)
  \hat{\psi}
  (4^{-j}\omega_1,4^jk\omega_1+2^{-j}\omega_2).
\end{equation*}
Hence,
$\hat{f}$ can be computed by simple multiplications
of the Fourier-transformed shearlet coefficients
with the dilated and sheared spectra of $\psi$ and
afterwards summing over all ``parts'', scales $j$ and all shears $k$, respectively.
In detail we have
\begin{align}\label{eq:inverseShearletTransform}
\begin{split}
  \hat{f}(\omega_1,\omega_2)
  =&
  \fftn(c(0,\,\cdot\,))
  \hat{\phi}(\omega_1,\omega_2)
  \\
  +&
  \sum_{j=0}^{j_0-1}
  \sum_{k=-2^j+1}^{2^j-1}
  \fftn(c(h,j,k,\cdot))
  \hat{\psi}
  (4^{-j}\omega_1,4^{-j}k\omega_1+2^{-j}\omega_2)
  \\
  +&
  \sum_{j=0}^{j_0-1}
  \sum_{k=-2^j+1}^{2^j-1}
  \fftn(c(v,j,k,\cdot))
  \hat{\psi}
  (4^{-j}\omega_2,4^{-j}k\omega_2+2^{-j}\omega_1)
  \\
  +&
    \sum_{j=0}^{j_0-1}
  \sum_{k=\pm 2^j}
  \fftn(c(h\times v,j,k,\cdot))
  \hat{\psi}
  (4^{-j}\omega_1,4^{-j}k\omega_1+2^{-j}\omega_2).
\end{split}
\end{align}
Finally, we get $f$ itself by $f = \ifftn(\hat{f})$.

\subsection{Smooth Shearlets}\label{sec:smooth_shearlets}
In many theoretical (and sometimes also practical) purposes one needs smooth shearlets in Fourier domain because such shearlets provide well-localized shearlets in time domain.
In \cite{GL13} a new shearlet construction is proposed that provides smooth shearlets for all scales $a$ and respective shears $s$.
Our shearlets are smooth for all scales and for all shears $\lvert s\rvert\neq 1$.
Our ``diagonal'' shearlets $\psi^{h\times v}$ are continuous by construction but they are not smooth. 
This is illustrated in Figure~\ref{fig:diagonalShearletOld}.
\begin{figure}[htbp]
  \centering
  \subfloat[Diagonal shearlet in our construction.]{\label{fig:diagonalShearletOld}
  \imageWithBorder{0.4\textwidth}{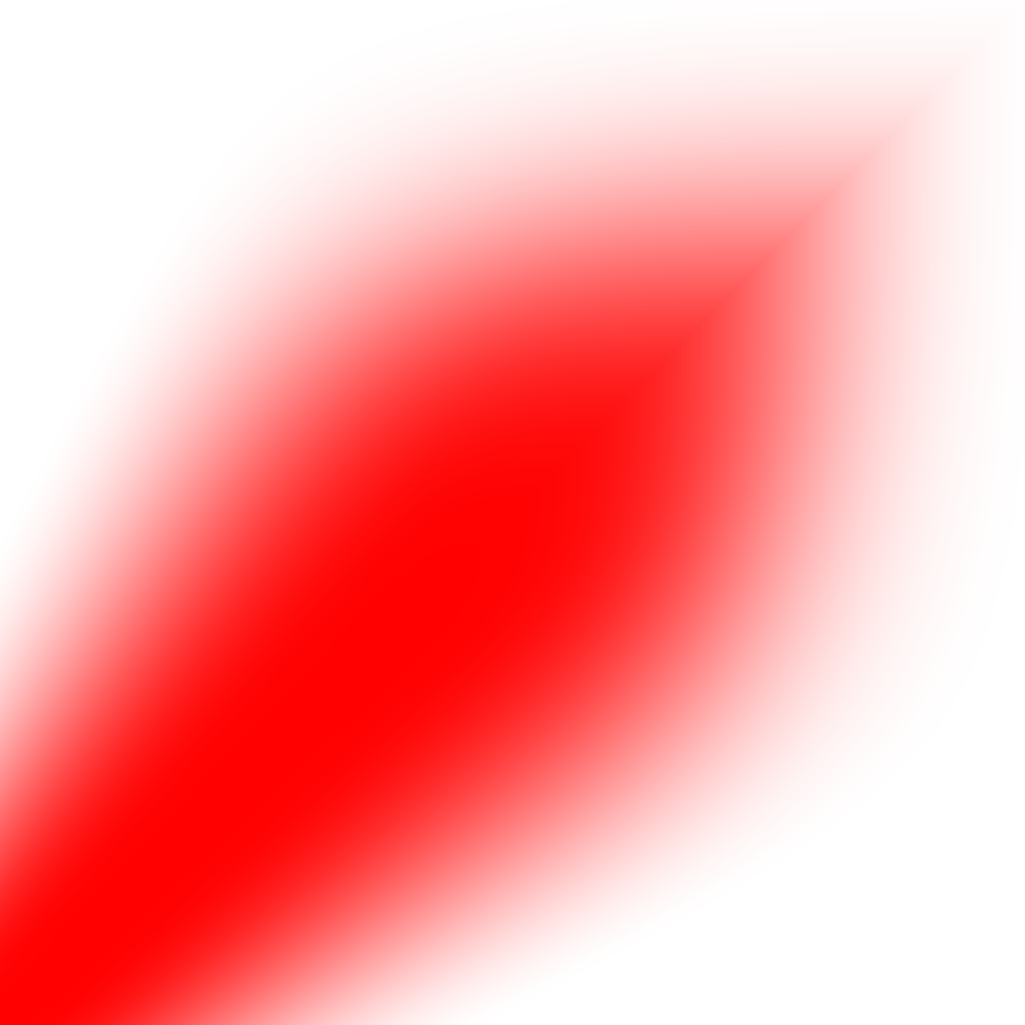}}
  \hspace{1cm}
  \subfloat[Diagonal shearlet in the new construction.]{\label{fig:diagonalShearletNew}
  \imageWithBorder{0.4\textwidth}{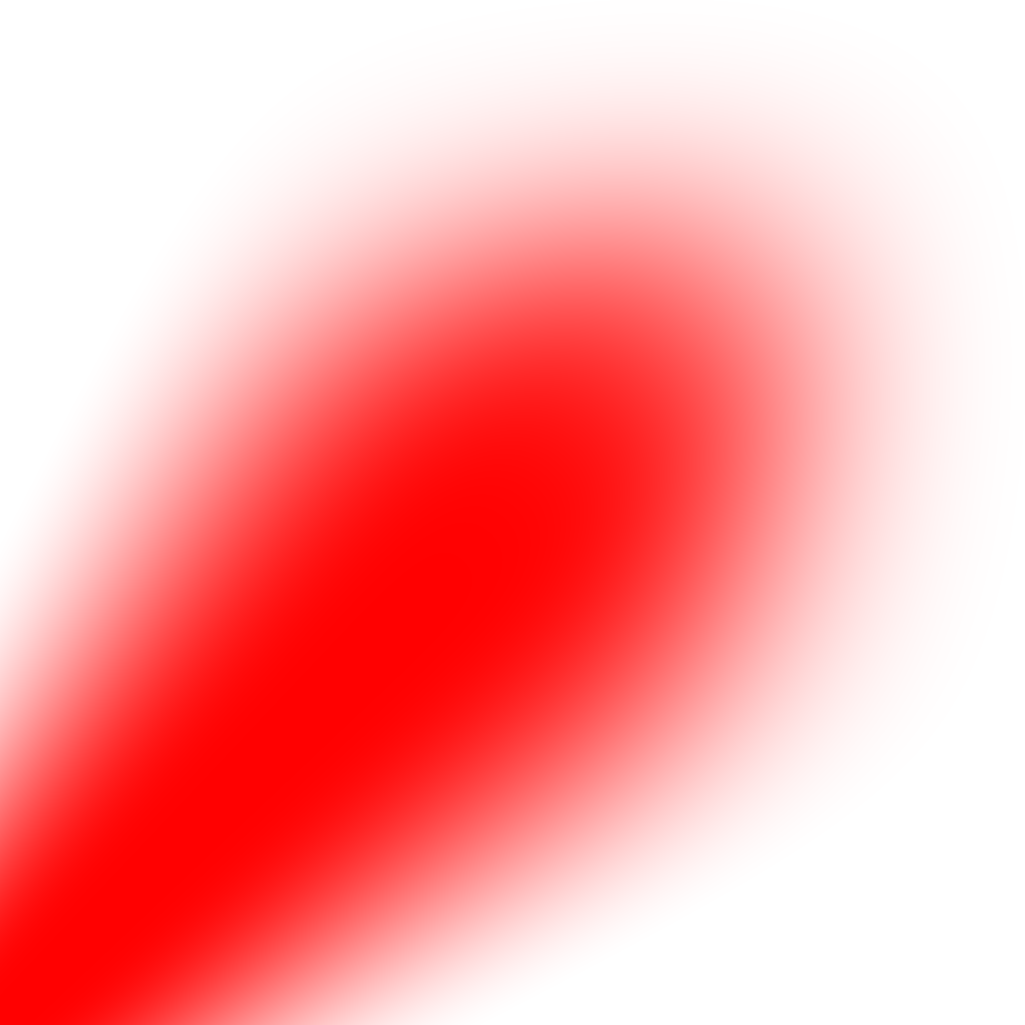}}
  \caption{Diagonal shearlets in our construction and in the new, smooth construction (Fourier domain).}
  \label{fig:diagonalShearletsOldVSNew}
\end{figure}

Obviously our construction is not smooth in points on the diagonal.
The new construction circumvents this with ``round'' corners.
To this end, we get back to the two different scaling functions which we discussed in Section~\ref{sec:scaling_function}.
While we chose the scaling function matching our cone-construction, the new construction is based on the tensor-product scaling function
$\hat{\Phi}(\omega)=\varphi(\omega_1)\varphi(\omega_2)$. 
We transfer the basic steps presented in \cite{GL13} to our setting. 
In fact, we only need to modify the function $\psi_1$.
We set
\begin{equation}\label{eq:newPsi1}\index{shearlet!smooth}
  \hat{\Psi}_1(\omega)
:=
  \sqrt{\hat{\Phi}^2(2^{-2}\omega_1,2^{-2}\omega_2)-\hat{\Phi}^2(\omega_1,\omega_2)}.
\end{equation}
Clearly, $\hat{\Psi}_1(\omega)$ fulfills $\sum_{j\geq 0}\hat{\Psi}_1^2(2^{-2j}\omega)=1$ for all
$\omega\in\Omega\setminus[-1,1]^2$.
We further have
\begin{equation*}
  \hat{\Phi}^2(\omega)
+
  \sum_{j\geq 0}\hat{\Psi}_1^2(2^{-2j}\omega)
=
  1
\quad
  \text{for all }\omega\in\Omega,
\end{equation*}
i.e., this setting provides also a Parseval frame.
Figure~\ref{fig:newPsi1} shows $\hat{\Psi}_1$.
Note that $\hat{\Psi}_1$ is supported in the Cartesian corona 
$[-4,4]^2\setminus [-\frac{1}{2},\frac{1}{2}]^2$.
\begin{figure}[htbp]
  \centering
  \imageWithBorder{0.4\textwidth}{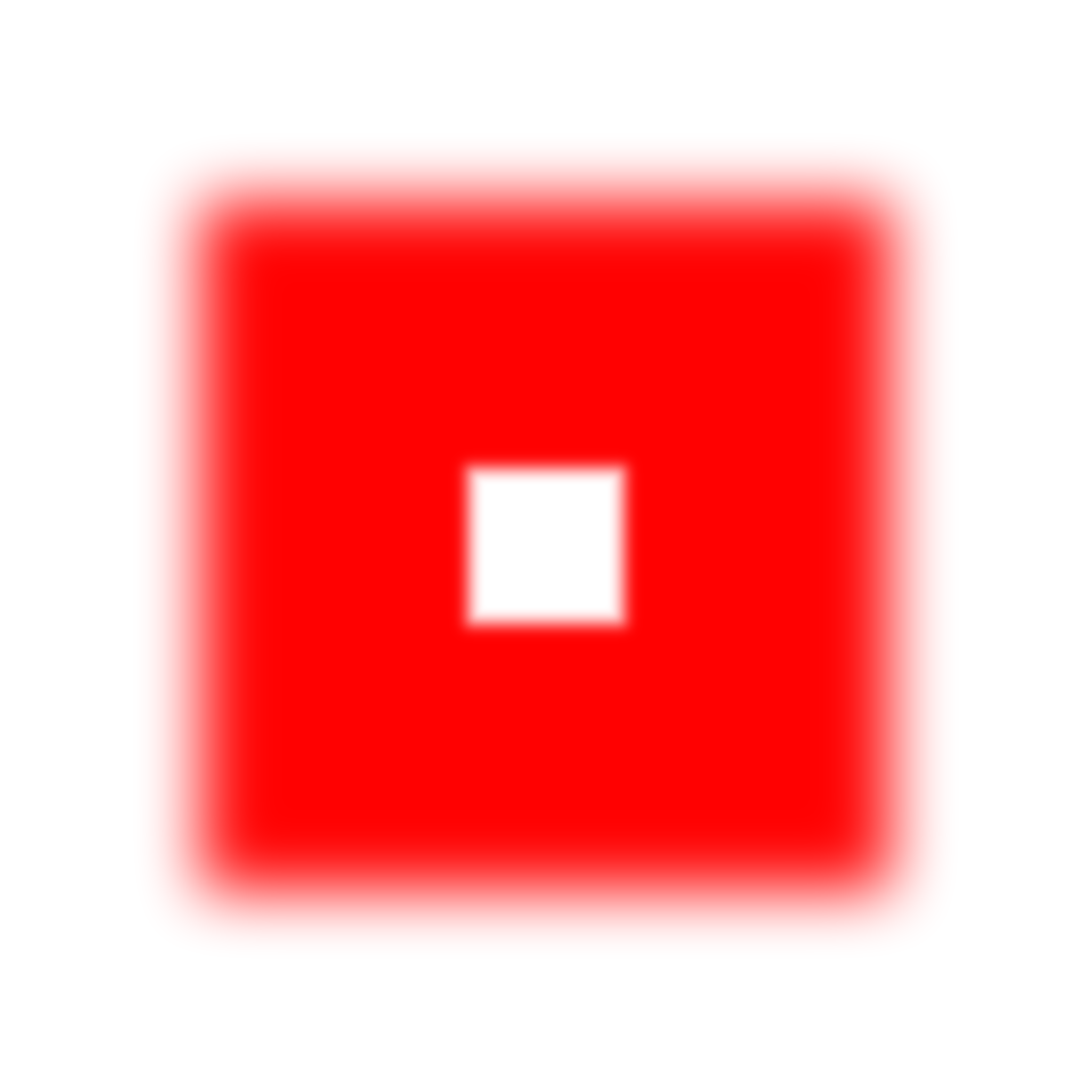}
  \caption[The new function Psi.]{The new function $\hat{\Psi}_1$ (see \eqref{eq:newPsi1}).}
  \label{fig:newPsi1}
\end{figure}
The full shearlet $\Psi$ reads similar as before:
\begin{equation} \label{eq:smoothPsi}
  \hat{\Psi}(\omega_1,\omega_2) 
= 
  \hat{\Psi}_1(\omega_1,\omega_2)\hat{\psi}_2\left(\frac{\omega_2}{\omega_1}\right).
\end{equation}
The construction of the horizontal, vertical and ``diagonal'' shearlets is the same as before, besides that the diagonal shearlets are smooth now, see Figure~\ref{fig:diagonalShearletNew}.

Before we examine the smoothness of the diagonal shearlets we discuss the differentiability of the remaining shearlets. 
Due to the construction we only need to analyze the functions $\hat{\psi}_1$ and $\hat{\psi_2}$. 
We have
\begin{equation*}
  \hat{\psi}_1(\omega_1)
=
  \sqrt{b^2(2\omega_1)+b^2(\omega_1)}
=
  \begin{cases}
  0                                                                             & \text{for }\lvert\omega_1\rvert \leq \frac{1}{2},  \\
  \sin\left(\frac{\pi}{2}v(2\lvert\omega_1\rvert-1)\right)                      & \text{for }\frac{1}{2} < \lvert\omega_1\rvert < 1, \\
  1                                                                             & \text{for }1 \leq \lvert\omega_1\rvert \leq 2,     \\
  \cos\left(\frac{\pi}{2}v\left(\frac{1}{2}\lvert\omega_1\rvert-1\right)\right) & \text{for }2 < \lvert\omega_1\rvert < 4 ,          \\
  0                                                                             & \text{for }\lvert\omega_1\rvert \geq 4
  \end{cases}
\end{equation*}
and with straightforward differentiation
\begin{equation*}
  \hat{\psi}_1'(\omega_1)
=
  \begin{cases}
  0                                                                             & \text{for }\lvert\omega_1\rvert \leq \frac{1}{2},  \\
  \pi 
  v'(2\lvert\omega_1\rvert-1)
  \cos\left(\frac{\pi}{2}v(2\lvert\omega_1\rvert-1)\right)                      & \text{for }\frac{1}{2} < \lvert\omega_1\rvert < 1, \\
  0                                                                             & \text{for }1 \leq \lvert\omega_1\rvert \leq 2,     \\
  -\frac{\pi}{2}  
  v'\left(\frac{1}{2}\lvert\omega_1\rvert-1\right)
  \sin\left(\frac{\pi}{2}v\left(\frac{1}{2}\lvert\omega_1\rvert-1\right)\right) & \text{for }2 < \lvert\omega_1\rvert < 4,           \\
  0                                                                             & \text{for }\lvert\omega_1\rvert \geq 4.
  \end{cases}
\end{equation*}
The derivative is continuous if and only if the values at the critical points 
$\{\tfrac{1}{2},1,2,4\}$ coincide (for symmetry reasons we can restrict ourselves to the positive range). 
We have  
$v'\left(2\cdot\frac{1}{2}-1\right) = v'(0) = v'\left(\frac{1}{2}\cdot 2-1\right)$ 
and 
$v'(2\cdot 1-1) = v'(1) v'\left(\frac{1}{2}\cdot 4-1\right)$.
Consequently, $\hat{\psi}_1'$ is continuous and in particular $\hat{\psi}_1\in C^1$ if and only if $v'(0) = 0 = v'(1)$.
By induction we see 
$\hat{\psi}_1^{(n)}\in C^{n}$ 
if and only if 
$v^{(n)}(0)=0$ and $v^{(n)}(1)=0$, $n\geq 1$.

For our $v$ in \eqref{eq:v} we have $v^{(3)}(1) = 0$ but $v^{(4)}(1)\neq 0$, i.e., $\hat{\psi}_1\in C^3$.

Similarly, we obtain for $\hat{\psi}_2$ that
\begin{equation*}
  \frac{\partial\hat{\psi}_2}{\partial \omega_1}(\omega_1,\omega_2)
=
  \begin{cases}
  -\frac{\omega_2}{\omega_1^2}v'(1+\frac{\omega_2}{\omega_1})\frac{1}{2\sqrt{v(1+\frac{\omega_1}{\omega_2})}} & \text{for }\frac{\omega_2}{\omega_1} \leq 0, \\
   \frac{\omega_2}{\omega_1^2}v'(1-\frac{\omega_2}{\omega_1})\frac{1}{2\sqrt{v(1-\frac{\omega_1}{\omega_2})}} & \text{for }\frac{\omega_2}{\omega_1} > 0,
  \end{cases}
\end{equation*}
where we see that $v'(0)=0$ in order for the derivative to exist. 
Thus, the shearlet $\hat{\psi}$ is $C^n$ if $v^{(n)}(0) = v^{(n)}(1) = 0$. 
This is also valid for the dilated and sheared shearlet $\hat{\psi}_{j,k,m}^h$ (and $\hat{\psi}_{j,k,m}^v$) for $\lvert k \rvert\neq 2^j$. 
We take a closer look at the diagonal shearlet for $k=-2^j$ where we have
\begin{equation*}
  \hat{\psi}_{j,-2^j,m}^{h\times v}(\omega)
=
  \begin{cases}
  \hat{\psi}_{j,-2^j,m}^h(\omega),      &\text{for }\omega\in\cone^h,     \\
  \hat{\psi}_{j,-2^j,m}^v(\omega),      &\text{for }\omega\in\cone^v,     \\
  \hat{\psi}_{j,-2^j,m}^\times(\omega), &\text{for }\omega\in\cone^\times.
  \end{cases}
\end{equation*}
Naturally, $\hat{\psi}_{j,-2^j,m}^{h\times v}$ is smooth for $\omega\in\cone^h$ and $\omega\in\cone^v$. 
Additionally, $\hat{\psi}_{j,-2^j,m}^{h\times v}(\omega)$ is continuous at the seam lines, 
but not differentiable there since we have for the partial derivatives of $\hat{\psi}_{j,-2^j,m}^{h}$ and $\hat{\psi}_{j,-2^j,m}^{v}(\omega)$ that
\begin{align*}
  \frac{\partial\hat{\psi}_{j,-2^j,m}^h}{\partial \omega_1}(\omega)
&=
  \hat{\psi}_2\left(2^j\left(\frac{\omega_2}{\omega_1} - 1 \right)\right)
  \e^{-\frac{2\pi i}{N}(\omega_1 m_1 + \omega_2 m_2)}\cdot
  2^{-2j}\frac{\partial\hat{\psi}_1}{\partial\omega_1}(2^{-2j}\omega_1)
  \\
&\quad  +
  \hat{\psi}_1(2^{-2j}\omega_1)
  \e^{-\frac{2\pi i}{N}(\omega_1 m_1 + \omega_2 m_2)}
  \left(-2^j\frac{\omega_2}{\omega_1^2}\right)
  \frac{\partial\hat{\psi}_2}{\partial\omega_1}\left(2^j\left(\frac{\omega_2}{\omega_1}-1\right)\right)
  \\
&\quad  +
  \hat{\psi}_1(2^{-2j}\omega_1)
  \hat{\psi}_2\left(2^j\left(\frac{\omega_2}{\omega_1}-1\right)\right)
  \left(-\frac{2\pi i}{N} m_1\right)
  \e^{-\frac{2\pi i}{N}\left(\omega_1 m_1 + \omega_2 m_2\right)}
\intertext{and}
  \frac{\partial\hat{\psi}_{j,-2^j,m}^v}{\partial \omega_1}(\omega)
&=
  \hat{\psi}_1(2^{-2j}\omega_2)
  \left(
  \e^{-\frac{2\pi i}{N}(\omega_1 m_1 + \omega_2 m_2)}
  \left(\frac{2^j}{\omega_2}\right)
  \frac{\partial\hat{\psi}_2}{\partial\omega_1}\left(2^j\left(\frac{\omega_1}{\omega_2}-1\right)\right)
  \right.
  \\
&\quad  +
  \left.
  \hat{\psi}_2\left(2^j\left(\frac{\omega_1}{\omega_2}-1\right)\right)
  \left(-\frac{2\pi i}{N} m_1\right)
  \e^{-\frac{2\pi i}{N}(\omega_1 m_1 + \omega_2 m_2)}
  \right).
\end{align*}
For $\omega_1 = \omega_2$ this reads
\begin{align*}
  \frac{\partial\hat{\psi}_{j,-2^j,m}^h}{\partial \omega_1}(\omega_1,\omega_1)
&=
  \e^{-\frac{2\pi i}{N}\omega_1 (m_1 + m_2)}
  \Biggl(
  \underbrace{\hat{\psi}_2(0)}_{=1}
  2^{-2j}\frac{\partial\hat{\psi}_1}{\partial\omega_1}(2^{-2j}\omega_1)
  \Bigr.
  \\
&\quad  -
  \hat{\psi}_1(2^{-2j}\omega_1)
  \left(\frac{2^j}{\omega_1}\right)
  \underbrace{\frac{\partial\hat{\psi}_2}{\partial\omega_1}(0)}_{=0}
  -
  \Bigl.
  \hat{\psi}_1(2^{-2j}\omega_1)
  \underbrace{\hat{\psi}_2(0)}_{=1}
  \left(\frac{2\pi i}{N} m_1\right)
  \Biggr)
  \\
&=
  \e^{-\frac{2\pi i}{N}\omega_1 (m_1 + m_2)}
  \left(
  2^{-2j}\frac{\partial\hat{\psi}_1}{\partial\omega_1}(2^{-2j}\omega_1)
  -
  \left(\frac{2\pi i}{N} m_1\right)
  \hat{\psi}_1(2^{-2j}\omega_1)
  \right)
\intertext{and}
  \frac{\partial\hat{\psi}_{j,-2^j,m}^v}{\partial \omega_1}(\omega_1,\omega_1)
&=
  \e^{-\frac{2\pi i}{N}\omega_1 (m_1 + m_2)}\!
  \Biggl(\!
  \hat{\psi}_1(2^{-2j}\omega_1)\!
  \left(\frac{2^j}{\omega_1}\right)\!
  \underbrace{\frac{\partial\hat{\psi}_2}{\partial\omega_1}(0)}_{=0}
  -
  \hat{\psi}_1(2^{-2j}\omega_1)
  \underbrace{\hat{\psi}_2(0)}_{=1}\!
  \left(\frac{2\pi i}{N} m_1\right)\!\!
  \Biggr)
  \\
&=
  \e^{-\frac{2\pi i}{N}\omega_1 (m_1 + m_2)}
  \left(
  -
  \left(\frac{2\pi i}{N} m_1\right)
  \hat{\psi}_1(2^{-2j}\omega_1)
  \right)
  .
\end{align*}
Obviously, both derivatives do not coincide, consequently, our shearlet construction is not smooth for the diagonal shearlets. 
Considering the new construction, we get for the both partial derivatives
\begin{align*}
  \frac{\partial\hat{\Psi}_{j,-2^j,m}^h}{\partial \omega_1}(\omega)
&=
  2^{-2j}
  \frac{\partial\hat{\Psi}_1}{\partial\omega_1}(2^{-2j}\omega)
  \hat{\psi}_2\left(2^j\left(\frac{\omega_2}{\omega_1} - 1 \right)\right)
  \e^{-\frac{2\pi i}{N}(\omega_1 m_1 + \omega_2 m_2)}
  \\
&\quad  -
  2^j\frac{\omega_2}{\omega_1^2}
  \hat{\Psi}_1(2^{-2j}\omega)
  \frac{\partial\hat{\psi}_2}{\partial\omega_1}\left(2^j\left(\frac{\omega_2}{\omega_1}-1\right)\right)
  \e^{-\frac{2\pi i}{N}(\omega_1 m_1 + \omega_2 m_2)}
  \\
&\quad  -
  \frac{2\pi i}{N} m_1
  \hat{\Psi}_1(2^{-2j}\omega)
  \hat{\psi}_2\left(2^j\left(\frac{\omega_2}{\omega_1}-1\right)\right)
  \e^{-\frac{2\pi i}{N}(\omega_1 m_1 + \omega_2 m_2)}
\intertext{and}
  \frac{\partial\hat{\Psi}_{j,-2^j,m}^v}{\partial \omega_1}(\omega)
&=
  2^{-2j}
  \frac{\partial\hat{\Psi}_1}{\partial\omega_1}(2^{-2j}\omega)
  \hat{\psi}_2\left(2^j\left(\frac{\omega_1}{\omega_2} - 1 \right)\right)
  \e^{-\frac{2\pi i}{N}(\omega_1 m_1 + \omega_2 m_2)}
  \\
&\quad  +
  \left(\frac{2^j}{\omega_2}\right)
  \hat{\Psi}_1(2^{-2j}\omega)
  \frac{\partial\hat{\psi}_2}{\partial\omega_1}\left(2^j\left(\frac{\omega_1}{\omega_2}-1\right)\right)
  \e^{-\frac{2\pi i}{N}(\omega_1 m_1 + \omega_2 m_2)}
  \\
&\quad  -
  \frac{2\pi i}{N} m_1
  \hat{\Psi}_1(2^{-2j}\omega)
  \hat{\psi}_2\left(2^j\left(\frac{\omega_1}{\omega_2}-1\right)\right)
  \e^{-\frac{2\pi i}{N}(\omega_1 m_1 + \omega_2 m_2)}.
\end{align*}
With $\omega_1=\omega_2$ we compute further
\begin{align*}
  \frac{\partial\hat{\Psi}_{j,-2^j,m}^h}{\partial \omega_1}(\omega_1,\omega_1)
&=
  2^{-2j}
  \frac{\partial\hat{\Psi}_1}{\partial\omega_1}(2^{-2j}\omega_1,2^{-2j}\omega_1)
  \underbrace{\hat{\psi}_2(0)}_{=1}
  \e^{-\frac{2\pi i}{N}\omega_1 (m_1 + m_2)}
  \\
&\quad  -
  2^j\frac{\omega_2}{\omega_1^2}
  \hat{\Psi}_1(2^{-2j}\omega_1,2^{-2j}\omega_1)
  \underbrace{\frac{\partial\hat{\psi}_2}{\partial\omega_1}(0)}_{=0}
  \e^{-\frac{2\pi i}{N}\omega_1 (m_1 + m_2)}
  \\
&\quad  -
  \frac{2\pi i}{N} m_1
  \hat{\Psi}_1(2^{-2j}\omega_1,2^{-2j}\omega_1)
  \underbrace{\hat{\psi}_2(0)}_{=1}
  \e^{-\frac{2\pi i}{N}\omega_1 (m_1 + m_2)}
\intertext{and}
  \frac{\partial\hat{\Psi}_{j,-2^j,m}^v}{\partial \omega_1}(\omega_1,\omega_1)
&=
  2^{-2j}
  \frac{\partial\hat{\Psi}_1}{\partial\omega_1}(2^{-2j}\omega_1,2^{-2j}\omega_1)
  \underbrace{\hat{\psi}_2(0)}_{=1}
  \e^{-\frac{2\pi i}{N}\omega_1 (m_1 + m_2)}
  \\
&\quad  +
  \left(\frac{2^j}{\omega_2}\right)
  \hat{\Psi}_1(2^{-2j}\omega_1,2^{-2j}\omega_1)
  \underbrace{\frac{\partial\hat{\psi}_2}{\partial\omega_1}(0)}_{=0}
  \e^{-\frac{2\pi i}{N}\omega_1 (m_1 + m_2)}
  \\
&\quad  -
  \frac{2\pi i}{N} m_1
  \hat{\Psi}_1(2^{-2j}\omega_1,2^{-2j}\omega_1)
  \underbrace{\hat{\psi}_2(0)}_{=1}
  \e^{-\frac{2\pi i}{N}\omega_1 (m_1 + m_2)}.
\end{align*}
It can be easily seen that both derivatives coincide if and only if $\frac{\partial\hat{\psi}_2}{\partial\omega_1}(0)=0$ since the second term vanishes. 
The same result is obtained for the partial derivative with respect to $\omega_2$. 
Consequently, the new construction is smooth everywhere.

\begin{remark}\label{remark:construction_v}
As we have seen the smoothness of the shearlets depends strongly on the smoothness of the function $v$. 
The function $v$ we have used was constructed to provide shearlets in $C^3$. 
The first three derivatives at $0$ and $1$ should be equal to zero, i.e., $v'(x) = c x^3(x-1)^3$. 
With $v(1)=1$ and straightforward integration we obtain $c=-140$ and the function $v$ as in \eqref{eq:v}.

Higher grades of smoothness are easily constructed with a new function $v$ by setting $v'(x) = c x^k(x-1)^k$. 
These shearlets would be in $C^k$.
To obtain shearlets in $C^\infty$ one needs another function $v$ with $v^{(n)}(0) = 0 = v^{(n)}(1)$ for all $n\geq 1$. 
The authors of \cite{MP10} propose
\begin{equation*}
  v(x)
=
  \frac{s(x-1)}{s(x-1) + s(x)}
\quad\text{where}\quad
  s(x)
=
  \e^{-\left(\frac{1}{(1+x)^2} + \frac{1}{(1-x)^2}\right)}.
\end{equation*}
\end{remark}

Note that due to our discretization $t=m$ we have a unique handling of both the horizontal and the vertical cone and do not have to make any adjustments for the diagonal shearlets.
This is in contrast to the discretization $t=A_{a_j}S_{s_{jk}}m$ where one has different discretizations for $t$ in the horizontal and in the vertical cone. 
Consequently, some adjustments for the diagonal shearlets are necessary.

Smooth shearlets are well-located in time. 
To show the difference we present in 
Figure~\ref{fig:diagonalShearletTime_old} the ``old'' shearlet in time domain and in comparison in 
Figure~\ref{fig:diagonalShearletTime_new} the new construction in time domain.
\begin{figure}[htbp]
  \centering
  \subfloat[Diagonal shearlet in our construction.]{\label{fig:diagonalShearletTime_old}
  \imageWithBorder{0.4\textwidth}{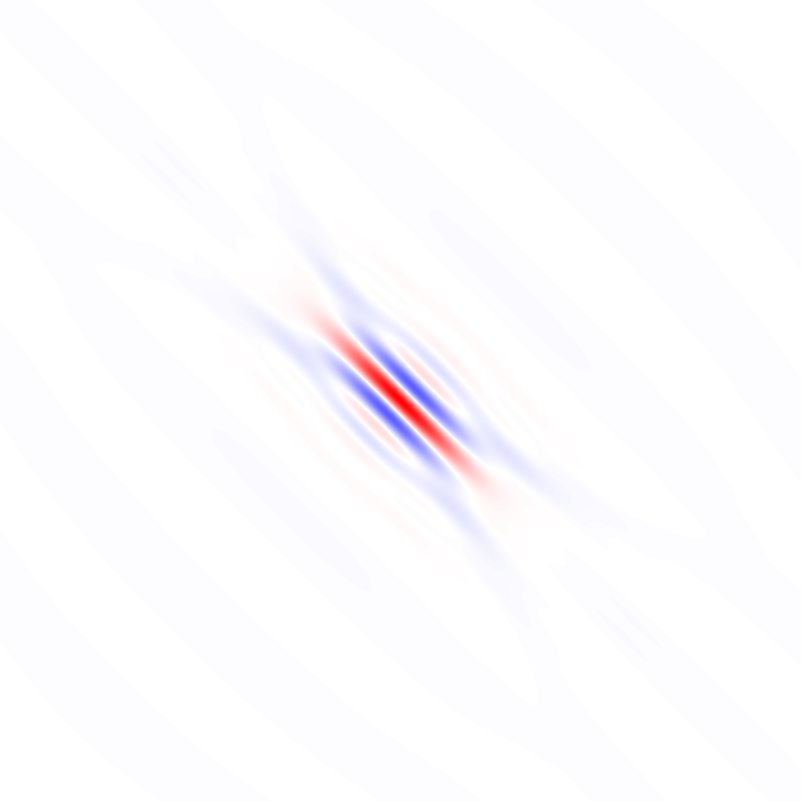}}
  \hspace{1cm}
  \subfloat[Diagonal shearlet in the new construction.]{\label{fig:diagonalShearletTime_new}
  \imageWithBorder{0.4\textwidth}{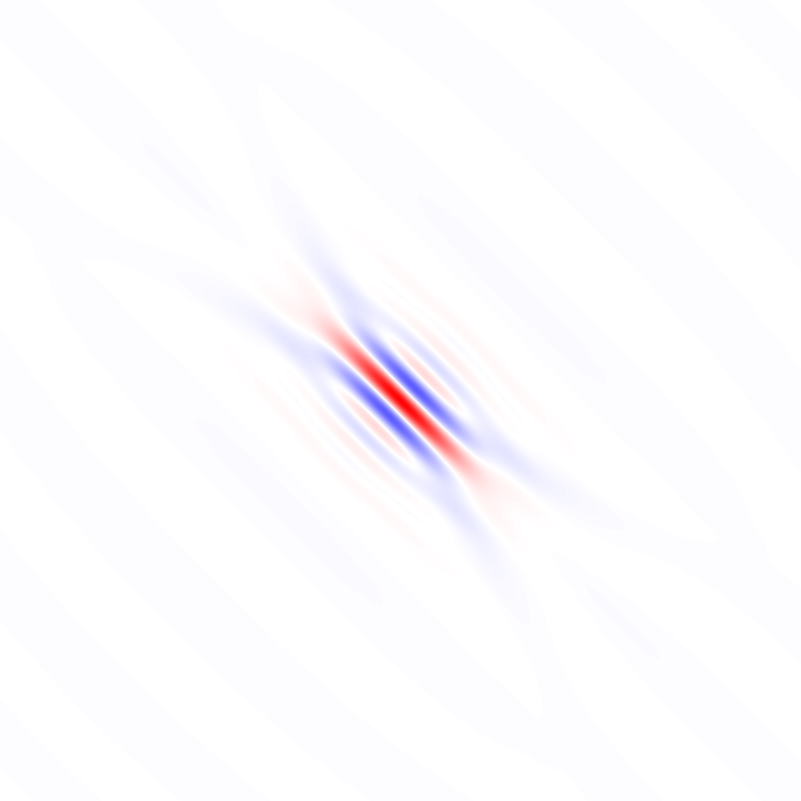}}
  \caption{Diagonal shearlets in our construction and in the new, smooth construction (time domain).}
  \label{fig:diagonalShearletsTimeOldVSNew}
\end{figure}
The non-smooth construction is slightly worse located. 
The shearlet coefficients of, e.g., a diagonal line, show only marginal differences such that for most practical applications it is irrelevant which construction is used.

\subsection{FFST: Fast Finite Shearlet Transform}
The implementation of the shearlet transform follows very closely the details described in the previous sections.
As we see in \eqref{eq:shearletTransform} and \eqref{eq:inverseShearletTransform}
for both the transform and the inverse transform the spectra of $\psi$ and $\phi$ are needed for all scales $j$ and all shears $s$ on
``all'' sets.
We precompute these spectra to use them for both directions of the transform.

Having the spectra the shearlet coefficients can easily be computed using \eqref{eq:shearletTransform} 
and also the reconstruction for given coefficients is straightforward using \eqref{eq:inverseShearletTransform}.

We will also comment on the efficient (or at least easily accessible) storage of the computed coefficients.

\subsubsection{Computation of Spectra}
We compute the spectra $\hat{\psi}_{j,k,m}$ as discrete versions of the continuous functions, i.e., 
we compute the values on a finite discrete lattice $\Xi \subset [-Y,Y] \times [-X,X]$ of size $M\times N$. 
Let without loss of generality $X\geq Y$ such that we focus on $X$ in the following.
The resulting $M\times N$ matrix will be element-wise multiplied with the Fourier transform of the given image.
This image is given as samples on (another) grid $\Omega$ and it might not be reasonable to choose $\Xi = \Omega$. To the contrary we will compute $\Xi$ and interpret the given point of the image as evaluated on the grid $\Xi$.

The choice of the range of the grid is not straightforward. 
Since we only consider finite images and a finite number of scales we have to ensure that the frame property remains valid, i.e., 
for all $\omega$ the sum $\sum_{j = 0}^{j_0-1}\lvert\hat{\psi}_1(2^{-2j}\omega)\rvert^2$ still equals $1$.

Recall that for $\omega\geq 0$: $\supp\hat{\psi}_1(\omega) = [\frac{1}{2},4] = [2^{-1},2^2]$ and $\hat{\psi}_1 \equiv 1$ for $\omega\in [1,2]=[2^0,2^1]$.
For the scaled version we further have $\supp\hat{\psi}_1(2^{-2j}\omega) = [2^{2j-1},2^{2j+2}]$ and $\hat{\psi}_1(2^{-2j}\omega)= 1$ for $\omega\in [2^{2j},2^{2j+1}]$.
We obtain
\begin{equation}\label{eq:sum_over_finite_scales}
  \sum_{j = 0}^{j_0-1}
  \lvert\hat{\psi}_1(2^{-2j}\omega)\rvert^2
  =
  \begin{cases}
    0                                                        & \text{for } \lvert\omega\rvert\leq \frac{1}{2},                \\
    \sin^2\left(\frac{\pi}{2}v(2\omega-1)\right)             & \text{for } \frac{1}{2}<\lvert\omega\rvert<1,                  \\
    1                                                        & \text{for } 1\leq \lvert\omega\rvert\leq 2^{2(j_0-1)+1},       \\
    \cos^2\left(\frac{\pi}{2}v(2^{-2j_0-1}\omega - 1)\right) & \text{for } 2^{2(j_0-1)+1}<\lvert\omega\rvert< 2^{2(j_0-1)+2}, \\
    0                                                        & \text{for } \lvert\omega\rvert\geq 2^{2(j_0-1)+2}.
  \end{cases}
\end{equation}
Thus, the sum is equal to 1 in a wide range of $\omega$.
As described above the part for $\lvert\omega\rvert<1$ where the sum increases from $0$ to $1$ matches with the decreasing part of the scaling function (compare \eqref{eq:overlapPsi1Phi}).
But we also have a decay for $\lvert\omega\rvert> 2^{2(j_0-1)+1}$ without compensation to $1$ due to the limited number of scales.
Keeping this decay would violate the frame property.

Figure~\ref{fig:shearlet_overlap} shows the dilated $\hat{\psi}_1$ for the highest considered scale $j=j_0-1$ (solid line) and for the two neighboring scales $j=j_0-2$ (dotted line) and $j=j_0$ (dashed line).
\begin{figure}[htbp]
  \centering
  \includegraphics[width=0.8\textwidth]{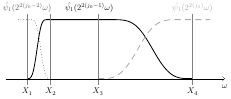}
  \caption{Overlapping shearlets.}
  \label{fig:shearlet_overlap}
\end{figure}
Now, the question is which of these $\omega$ our grid $\Xi$ should cover. 
We have marked four possible points $X_i$, $i=1,\ldots,4$, that we want to discuss in detail. 
The last one, $X_4$ might be a natural choice since the largest selected scale would be completely covered. 
However, this would destroy the frame property as the decay between $X_3$ and $X_4$ is not compensated by a higher scale.
Possible and recommended  is any choice between $X_2$ and $X_3$. The latter one is the last point for which the shearlet is equal to $1$ 
and the first one $X_2$ is the first point where the shearlet is equal to $1$. 
Choosing $X_3$ provides the largest possible finest scale whereas $X_2$ provides the smallest (reasonable) finest scale. 
From a theoretical point of view any point between $X_1$ and $X_2$ is also possible but in this case the finest scale would be really small and additionally the second finest scale would also get smaller. Finally, choosing $X_1$ would reduce the number of scales since the finest scale is not considered at all.

Following \eqref{eq:sum_over_finite_scales} the chosen $X$ must be less or equal than $2^{2(j_0-1)+1} = 2^{2j_0-1}$ ($X_3$ in Figure~\ref{fig:shearlet_overlap}) for the decay to be ``outside'' the image and on the other hand $X$ must be greater or equal than $2^{2(j_0-2)+2} = 2^{2j_0-2}$ ($X_2$ in Figure~\ref{fig:shearlet_overlap}).
To analyze the relation between grid size and image size we set $X=2^{2j_0-1}$.

To compute the grid and the spectra we assume that $M=2m+1$ and $N=2n+1$ are odd.
Then, we have a symmetric grid around $0$, hence, we have $n$ (respectively $m$) grid points in the negative range and $n$ (respectively $m$) grid points in the positive range and one grid point at $0$.
If the given $N$ (or $M$) is even we increase it by $1$. 
After computing grid and spectra we neglect the last row and/or column to retain the original image size.
We compute the number of considered scales based on the larger dimension. 
This leads to rectangular frequency bands. Without loss of generality we assume that $N=\max\{M,N\}$.
Having $n=\frac{N-1}{2}$ grid points for the positive range and the maximal distance between two grid points $\Delta=1$ we get
\begin{equation*}
  X = 2^{2j_0-1}=\frac{N-1}{2}
\Longrightarrow
  j_0 = \frac{1}{2}\log_2(N-1).
\end{equation*}
We set for the number of scales (as used above) $j_0:=\lfloor \frac{1}{2}\log_2(N)\rfloor$.
In the following table we list the number of scales for all image sizes $N=4\ldots,1024$:
\begin{center}
\begin{tabular}{l|c|c|c|c|c}
$N$ & $4,\ldots,15$ & $16,\ldots,63$ & $64,\ldots,255$ & $256,\ldots,1023$ & $1024$  \\\hline
$j_0$ & 1 & 2 & 3 & 4 & 5
\end{tabular}.
\end{center}

With $j_0$ fixed we can compute the distance $\Delta$ between two points.
As we have seen the largest value in the grid should be $X=2^{2j_0-1}$. 
For an odd $N$ the grid ranges from $[-X,X]$ and for an even grid we have the range $[-X,X) = [-X,X-\Delta]$. 
We assume again an odd $N$, such that the interval $[-X,X]$ should be divided in $N$
grid points including the bounds $-X$ and $X$ leading to $N-1$ subintervals and
\begin{equation*}
  \Delta
=
  \frac{2\cdot X}{N-1}
=
  \frac{2\cdot 2^{2j_0-1}}{N-1}
=
  \frac{2^{2j_0}}{N-1}
\end{equation*}
where $\Delta=1$ if $N=2^{2j_0}+1$ and $\Delta > \frac{1}{4}$, i.e., $\frac{1}{4}<\Delta\leq 1$. 
Thus, for the same number of scales we obtain a better resolution with increasing image size.

It seems a little awkward to discretize $\hat{f}$ and $\hat{\psi}$ on different lattices. 
However, with this auxiliary construction the definition and properties of the shearlet $\hat{\psi}$ are much more convenient. 
Additionally, the shearlets are now independent of the parameter $\Delta$ (or other grid properties). 
Anyway, to circumvent the imperfection with two lattices we could formerly also discretize $\hat{\psi}(\Delta\omega)$ on $\Omega$ instead of $\hat{\psi}(\omega)$ on $\Xi$ and obtain the same spectra.

\subsubsection{Indexing}\label{section:indexing}
To reduce the number of parameters
we introduce one index $i$ which replaces the parameters $\kappa$, $j$ and $k$.
We set $i=1$ for the low-pass part. We continue with the lowest frequency band, i.e., $j=0$.
The different cones and shear parameters represent the different directions of the shearlet.
Imagine the shearlet in Fourier domain to be a line which is rotated counter-clockwise around the center and assign the index $i$ accordingly.
In each frequency band we start in the horizontal position,
i.e., $\kappa = h$ and $k = 0$, and increase $i$ by one.
For each $k = -1,\ldots, -2^j+1$ we continue increasing the index by one.
The line is now almost in a 45\textdegree\ angle (or a line with slope $1$).
The next index is assigned to the combined shearlet
``$h\times v$'' at the seam line which covers the ``diagonal'' for $k=-2^j$.
We continue in the vertical cone for $k=-2^j+1,\ldots,2^j-1$.
Next is again the combined shearlet for $k=2^j$.
With decreasing shear, i.e., $k=2^j,\ldots,1$,
we finish the indexing for this frequency band and continue with the next one.
Figure~\ref{fig:frequency_tiling_indices} illustrates the indices for the first two scales.
\begin{figure}[htbp]
  \centering
  \includegraphics[width=0.6\textwidth]{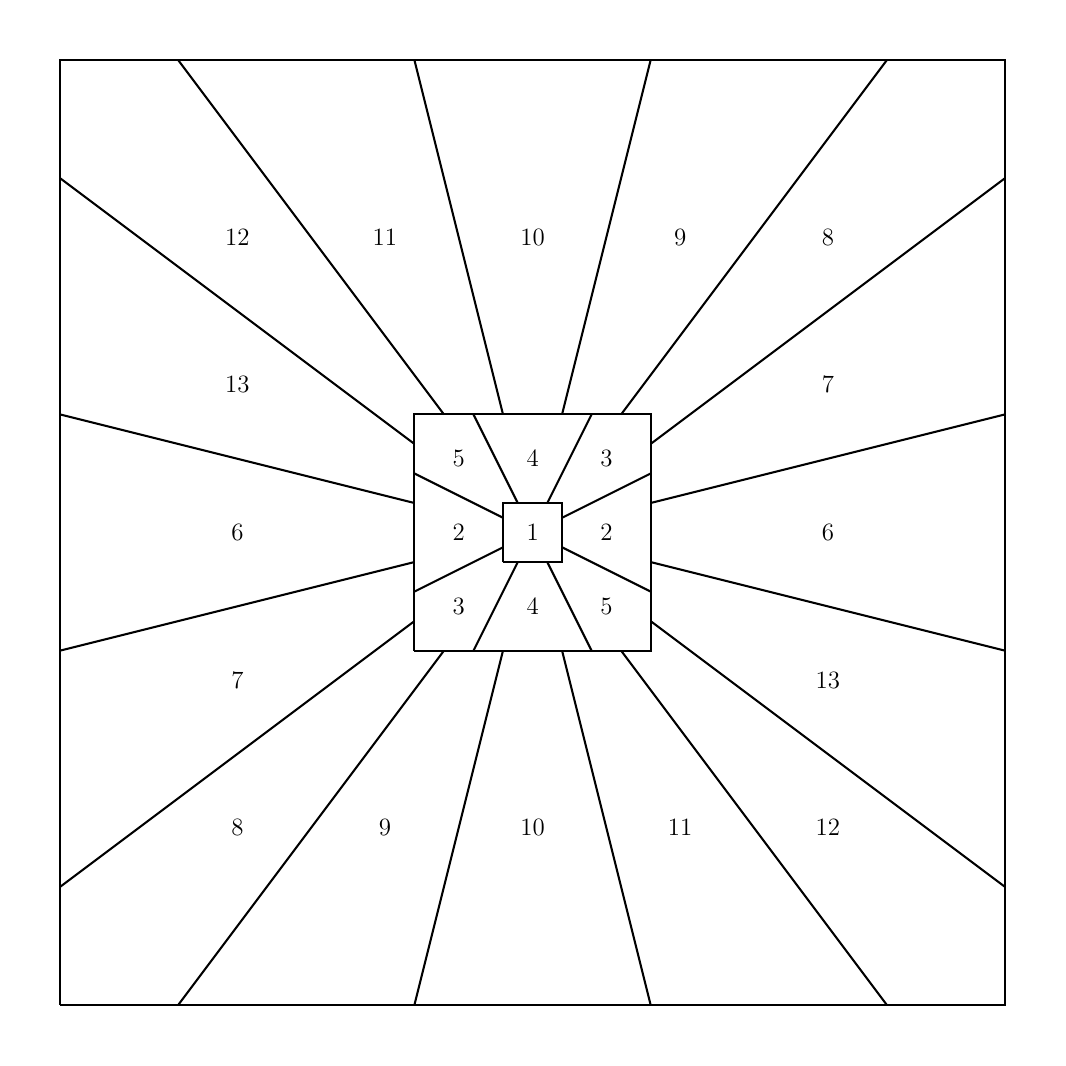}
  \caption{Frequency tiling with respective indices $i$.}
  \label{fig:frequency_tiling_indices}
\end{figure}

Summarizing the described procedure we always have one index for the low-pass part.
In each frequency band we have two indices (or shearlets) for the diagonals ($k=\pm 2^j$)
and in each cone we have $1 + 2\cdot(2^j-1) = 2^{j+1}-1$ shearlets.
For scale $j$ we have $2\cdot(2^{j+1}-1) + 2 = 2^{j+2}$ shearlets.
The following table lists the number of shearlets for each $j$:
\begin{center}
  \begin{tabular}{c|c|c|c}
    low-pass & $j=0$ & $j=1$ & $j=2$ \\ \hline
    1        & 4     & 8     & 16    
  \end{tabular}.
\end{center}
With a maximum scale $j_0-1$ the number of all indices $\eta$ is
\begin{equation*} 
  \eta 
= 
  1 + \sum_{j=0}^{j_0-1} 2^{j+2} 
= 
  1 + 4\sum_{j=0}^{j_0-1}2^j 
= 
  1 + 4\cdot (2^{j_0}-1) 
= 
  2^{j_0+2} - 3.
\end{equation*}
For each index the spectrum is computed on a grid of size $M\times N$.
We store all indices in a three-dimensional matrix of size $M\times N \times \eta$.
The first both components refer to the $\omega_2$ and $\omega_1$ coordinates and the third component is the respective index.
Consequently,
an image $f$ of size $M\times N$ is oversampled to an image of size $M\times N\times \eta$.
In particular we have an oversampling factor of $\eta$. 
The following table shows $\eta$ for $j_0=1,\ldots,4$:
\begin{center}
  \begin{tabular}{l|c|c|c|c}
    $j_0$  & 1 &  2 &  3 &  4  \\\hline
    $\eta$ & 5 & 13 & 29 & 61 
  \end{tabular}.
\end{center}
Note that $j_0$ is the \emph{number} of scales, the highest scale parameter $j$ is always $j_0-1$, i.e., we have the scale parameters $0,\ldots,j_0-1$. 
The function \lstinline{helper/shearletScaleShear}
provides various possibilities to compute the index $i$ from $a$ and $s$ or from $j$ and $k$ and vice versa. 
See the documentation inside the file for more information.

Figure~\ref{fig:shearlet_stack} shows shearlet coefficients and spectra stored as the described stack.
\begin{figure}[htbp]
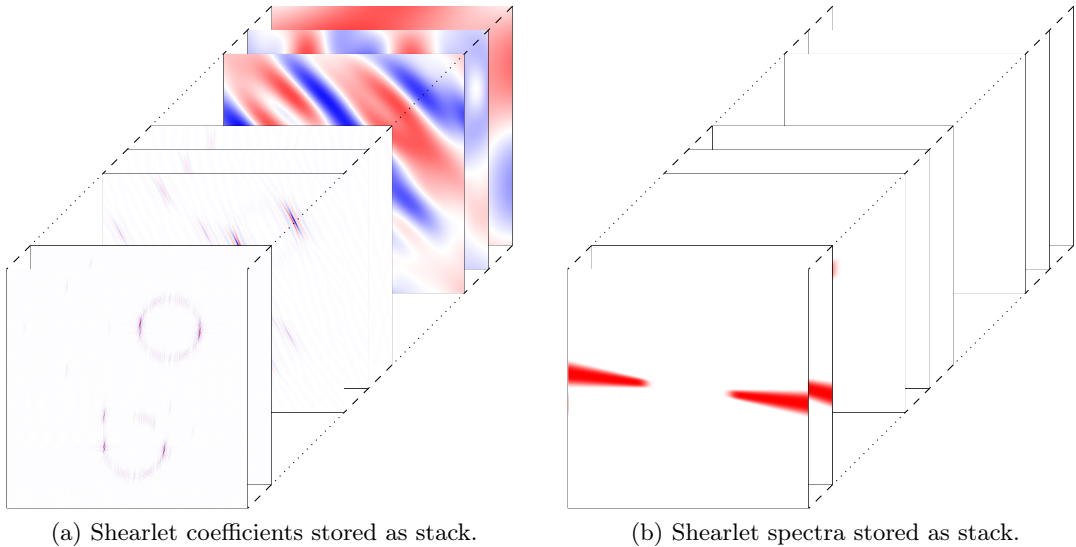

  \centering
  \subfloat[Shearlet coefficients stored as stack.]{\label{fig:shearlet_coefficients_stack}
    \input{images/indexing_coefficients.tikz}
  }
  \hspace{0.2cm}
  \subfloat[Shearlet spectra stored as stack.]{\label{fig:shearlet_spectra_stack}
    \input{images/indexing_spectra.tikz}
  }
  \caption{Shearlets and shearlet coefficients stored as stack.}
  \label{fig:shearlet_stack}
\end{figure}
A useful overview over the shearlet coefficients $c\in\RR^{M\times N\times \eta}$
is the mean value for each for each index $i$ over all translations, i.e., we compute the vector $d\in\RR^\eta$ as
\begin{equation}\label{eq:shearlet_meanvalue}
  d(i)
=
  \frac{1}{MN}
  \sum_{m=1}^M
  \sum_{n=1}^N
  \lvert c(m,n,i) \rvert,
\quad
  i=1,\ldots,\eta.
\end{equation}
For the test image in Figure~\ref{fig:formen} $d$ is shown in Figure~\ref{fig:shearlet_meanvalue}.
The dashed vertical lines represent the different scales, starting with the low-pass on the left and on the right the finest scale. The horizontal, vertical and diagonal directions are symbolized by a small rotated bar at the horizontal axis. We see that all directions appear with similar value in the second coarsest scale---due to the circle. The small peaks in both finest scale for both diagonal directions are due to the diamond.
\begin{figure}[htbp]
  \centering
  \includegraphics[width=0.8\textwidth]{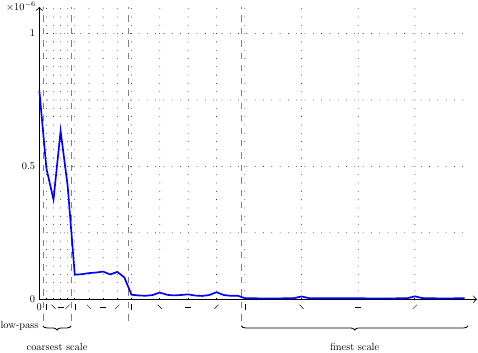}
  \caption{Mean value of shearlet coefficients for each index $i$ (scale and shear).}
  \label{fig:shearlet_meanvalue}
\end{figure}

\subsubsection{Short Documentation}
Every file contained in the package is commented, see there for details on the arguments, return values and examples. 
We only want to comment on the two most important functions.

The transform for an image \lstinline{A}
$\in\RR^{M\times N}$ is called with the following command
\begin{lstlisting}
  [ST,Psi] = shearletTransformSpect(A,numOfScales,realCoefficients)
\end{lstlisting}
where 
\lstinline{numOfscales}
and
\lstinline{realCoefficients}
are optional arguments.
If not given the number of scales $j_0$ is computed from the size of \lstinline{A},
i.e., $j_0 = \lfloor \frac{1}{2}\log_2(\max\{M,N\})\rfloor$
As default real shearlets are computed using 
the shearlet defined in \eqref{eq:Psi1} and \eqref{eq:Psi2}.
On the other hand \lstinline{numOfScales}
can be used two-fold.
If given as a scalar value it simply states the number of scales to consider.
On the other hand we can provide precomputed shearlet spectra which are then used for the computation of the transform.

The variable \lstinline{ST}
contains the shearlet coefficients as a three-dimensional matrix of size $M\times N \times \eta$ with the third dimension ordered as described in section \ref{section:indexing}. 
\lstinline{Psi}
is of same size and contains the respective shearlet spectra $\hat{\psi}_{j,k,0}^\kappa$.

With the additional parameters \lstinline{shearletSpect}
and \lstinline{shearletArg}
other shearlets can be used to compute the spectra.
Included in the software is \lstinline{meyerShearletSpect}
as default shearlet (based on \eqref{eq:Psi1} and \eqref{eq:Psi2}) and \lstinline{meyerSmoothShearletSpect}
for the new smooth construction (see \eqref{eq:smoothPsi}).
The value of the parameters are strings or directly the respective function handle.
To compute shearlet coefficients using the smooth shearlets, call
\begin{lstlisting}
  [ST,Psi] = shearletTransformSpect(A,numOfScales,realCoefficients,'shearletSpect',@meyerSmoothShearletSpect)
\end{lstlisting}
The parameter \lstinline{shearletArg}
can be used in both cases to provide the function handle (or function name as string) of an alternative auxiliary function, see also \lstinline{examples.m}.

The usage of different shearlet spectra is straight forward. 
One the one hand one can simply compute them externally in the matrix \lstinline{Psi} 
and provide them as the parameter \lstinline{numOfScales}. 
On the other hand it is possible to provide an own function '\lstinline{myShearletSpect}'
(with arbitrary name) with the function head
\begin{lstlisting}
  Psi = Psi = meyerShearletSpect( x, y, a, s, realCoefficients, shearletArg, scaling)
\end{lstlisting}
that computes the spectrum \lstinline{Psi} 
for given (meshgrids) \lstinline{x} 
and \lstinline{y} 
for scalar scale \lstinline{a} 
and shear \lstinline{s}
and (optional) parameter \lstinline{shearletArg}. 
For \lstinline{scaling='scaling'} 
it should return the scaling function. 
To obtain a reasonable transform the shearlet should provide a Parseval frame.
To check this just compute (and plot) \lstinline{sum(abs(Psi).^2,3)-1}. 
The values should be close to zero (see Figure~\ref{fig:frameTightness}) and \lstinline{examples.m}. 
Call the shearlet transform with the new shearlet spectrum by setting the parameter \lstinline{sherletSpect} 
to 
\lstinline{@myShearletSpect} 
or whatever you chose as the name of your shearlet function.

Further parameters are \lstinline{realReal}
(default $1$) that guarantees real coefficients, see Section~\ref{section:complex_shearlets_realreal}
and \lstinline{maxScale}
(default \lstinline{'max'})
that controls the size of the finest scale (either \lstinline{'min'}
or
\lstinline{'max'}), 
see Figure~\ref{fig:shearlet_overlap}.

The inverse transform is called with the command
\begin{lstlisting}
  A = inverseShearletTransformSpect(ST,Psi)
\end{lstlisting}
for the shearlet coefficients \lstinline{ST}.
As the second argument the shearlet spectra \lstinline{Psi}
should be provided for faster computations, if not given, the spectra are computed with default values or given parameters (as for \lstinline{shearletTransformSpect.m}).


\subsubsection{Download \& Installation}
The \textsc{Matlab}-Version of the toolbox is available for free download at
\begin{center}
  \url{http://www.mathematik.uni-kl.de/imagepro/software/}
\end{center}
The zip-file contains all relevant files and folders. Simply unzip the archive and add the folder (with subfolders!) to your \textsc{Matlab}-path.

The folder \textbf{FFST} contains the main files for the two directions of the transform. 
The included shearlets are stored in the folder \textbf{shearlets}. 
The folder \textbf{helper} contains some helper functions. 
To create  simple geometric structures some functions are provided in \textbf{create}. 
See \lstinline{contents.m} 
and the comments in each file for more information.

The following listing shows the subdirectories and the respective files
\dirtree{%
.1 FFST/.
.2 create/.
.3 myBall.m.
.3 myPicture.m.
.3 myPicture2.m.
.3 myRhombus.m.
.3 mySquare.m.
.2 helper/.
.3 checkInputs/.
.4 checkCoefficients.m.
.4 checkImage.m.
.4 checkLength.m.
.4 checkNumOfScales.m.
.4 checkShearletSpect.m.
.4 defaultNumberOfScales.m.
.3 parseShearletParameterInputs.m.
.3 scalesShearsAndSpectra.m.
.3 shearletScaleShear.m.
.2 shearlets/.
.3 bump.m.
.3 meyeraux.m.
.3 meyerScaling.m.
.3 meyerShearletSpect.m.
.3 meyerSmoothShearletSpect.m.
.3 meyerWavelet.m.
.2 contents.m.
.2 examples.m.
.2 inverseShearletTransformSpect.m.
.2 shearletTransformSpect.m.
.2 simple\_example.m.
}

If everything is installed correctly run \lstinline{simple_example} for testing. The result should look like Figure~\ref{fig:simple_example}.
\begin{figure}[htbp]
	\centering
	\includegraphics[width=0.5\textwidth]{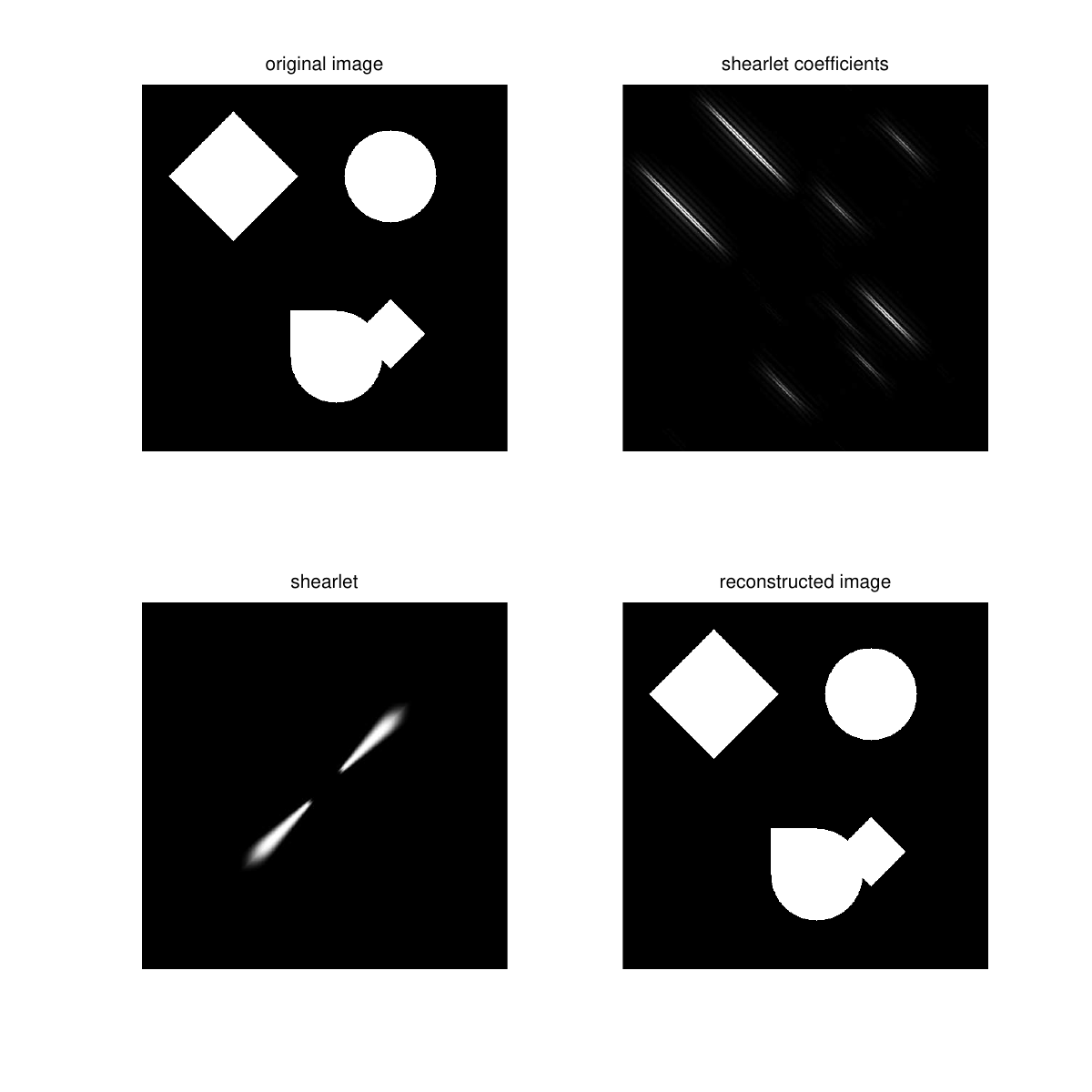}
	\caption{Result of script \lstinline{simple_example}.}
	\label{fig:simple_example}
\end{figure}

\subsection{Performance}
To evaluate the performance and the exactness of our implementation we present the following figures:
In Figure~\ref{fig:frameTightness} we investigate the numerical tightness of the frame.
The figure shows the difference between the square sum of the shearlets and $1$, i.e.,
\begin{equation*}
  \sum_{\kappa\in\{h,v\}}
  \sum_{j=0}^{j_0-1}
  \sum_{k=-2^j}^{2^j}
  \lvert
    \hat{\psi}^\kappa_{j,k,0}
  \rvert^2
  +
  \sum_{j=0}^{j_0-1}
  \sum_{k=\pm 2^j}
  \lvert
    \hat{\psi}^{h\times v}_{j,k,0}
  \rvert^2
  +
  \lvert
    \hat{\phi}_0
  \rvert^2
  -
  1.
\end{equation*}
The largest deviation is about $8\cdot 10^{-15}$ which is $40$ times the machine precision.
Figure~\ref{fig:transformExactness} shows the difference between a random image and after transform and inverse transform,
i.e., the exactness of the forward and backwards transform.
Here the biggest difference is about $2\cdot 10^{-15}$ or approximately 10 times the machine precision.
Surprisingly, this is even better than the tightness of the used frame.
\begin{figure}[htbp]
  \centering
  \subfloat[Frame tightness.]{\label{fig:frameTightness}
  \imageWithRedBlueColorbarExp{0.3\textwidth}{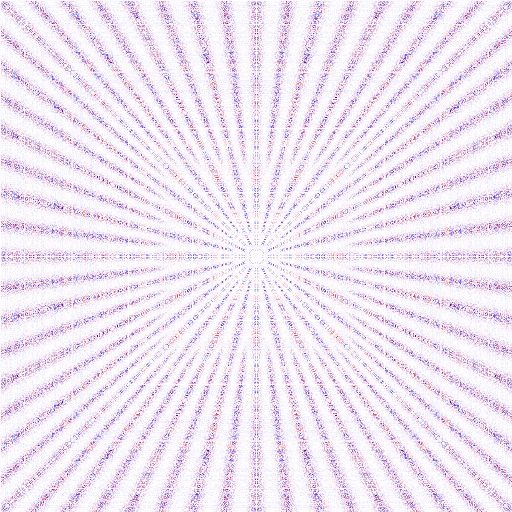}{9.10}{{0.2197,0.4394,0.6591,0.8788}}{-15}}  %
  \hspace{1.5cm}
  \subfloat[Transform exactness.]{\label{fig:transformExactness}
  \imageWithRedBlueColorbarExp{0.3\textwidth}{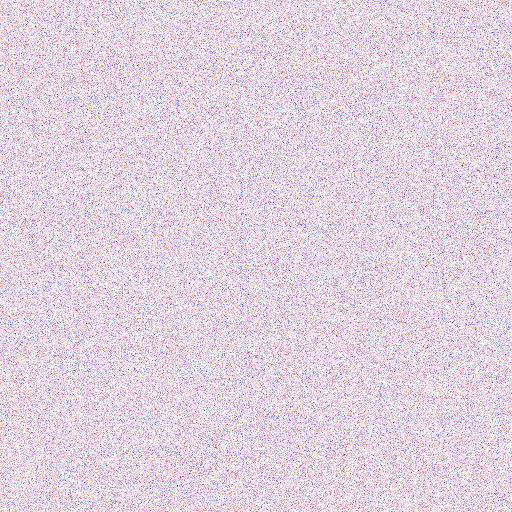}{2.22}{{0.2252,0.4504,0.6755,0.9007}}{-15}}  %
  \caption{Frame tightness and exactness of the implemented shearlet transform.}
  \label{fig:tightnessAndExactness}
\end{figure}

Most of the computation time is needed to precompute the spectra.
Having the spectra the transform and the inverse can be computed efficiently.
The running time of all mentioned implementations including ours is comparable
but depends strongly on the used frequency tiling and thus of the oversampling factor.

\subsection{Remarks}\label{section:Remarks}
\begin{enumerate}[(i)]
  \item 
    In \cite{KSZ12} and the respective implementation \emph{ShearLab}
    a pseudo-polar Fourier transform is used to implement a discrete (or digital) shearlet transform.
    For the dilation $a$ and shear $s$ the same discretization as before is used.
    But for the translation $t$ the authors set $t_{j,k,m}:=A_{a_j,\frac{1}{2}}S_{s_{j,k}}m$
    where we in contrast simply set $t_m:=m$ (see \eqref{eq:discrete_t}).
    Thus, their discrete shearlet becomes
    \begin{equation*}
      \hat{\widetilde{\psi}}_{j,k,m}(\omega)
    =
      \hat{\psi}
      \left(A_{a_j,\frac{1}{2}}S_{s_{j,k}}^\tT\omega\right)
      \e^{-2\pi \i\bigl\langle \omega, A_{a_j,\frac{1}{2}}S_{s_{j,k}}m \bigr\rangle}
    =
      \hat{\psi}
      \left(A_{a_j,\frac{1}{2}}S_{s_{j,k}}^\tT\omega\right)
      \e^{-2\pi \i\bigl\langle S_{s_{j,k}}^\tT A_{a_j,\frac{1}{2}}\omega, m \bigr\rangle}
      .
    \end{equation*}
    Since the operation $S_{s_{j,k}}^\tT A_{a_j,\frac{1}{2}}\omega$ would destroy the pseudo-polar grid a ``slight'' adjustment is made and the exponential term is replaced by
    \begin{equation*}
      \e^{
        -2\pi \i
        \bigl\langle 
          \bigl(\theta\circ S_{s_{j,k}}^{-\tT}\bigr)
          S_{s_{j,k}}^\tT A_{a_j,\frac{1}{2}}\omega, m 
        \bigr\rangle
      }
    \end{equation*}
    with $\theta\colon \RR\setminus\{0\}\times\RR \to \RR\times\RR$ and 
    $\theta(x,y) = (x,\frac{y}{x})$ such that
    \begin{equation*}
      \e^{
        -2\pi \i
        \bigl\langle 
          \bigl(\theta\circ S_{s_{j,k}}^{-\tT}\bigr)
          S_{s_{j,k}}^\tT A_{a_j,\frac{1}{2}}\omega, m 
        \bigr\rangle
      }
    =
      \e^{
        -2\pi \i
        \bigl\langle 
          \bigl(a_j\omega_1, \sqrt{a_j}\frac{\omega_2}{\omega_1}\bigr),
          m 
        \bigr\rangle
      }.
    \end{equation*}
    With this adjustment the last step of the shearlet transform can be obtained with a standard inverse fast Fourier transform
    (similar as in our implementation).
    Unfortunately, this is no longer related to translations of the shearlets in time domain.

  \item 
    We are aware of our larger oversampling factor in comparison with, e.g., \emph{ShearLab}.
    Having four scales we obtain 61 images of the same size as the original image.
    But since shearlets are designed to detect edges in images we like to avoid any down-sampling and keep translation invariance.
    A possibility to reduce the memory usage is to use the compact support of the shearlets in the frequency domain and only compute them on a ``relevant'' region.
    This approach is called \emph{wrapping} in \cite{CDDY06} and used in the respective implementation of CurveLab.
    But we then also have to store the position and size of each region which decreases the memory savings and makes the implementation a lot more complicated.
\end{enumerate}

\subsection{Complex Shearlets} 
\subsubsection{Guaranteeing Real Shearlet Coefficients}\label{section:complex_shearlets_realreal}
\index{shearlet!real}
The inner product between two real-valued functions or vectors is again real-valued. 
However, when computing the shearlet coefficients in the way described above, 
we obtain complex coefficients for the finest scale if the image is of even length.
Figure~\ref{fig:mean_value_imaginary_part_original} shows the mean value of the absolute value of the imaginary part for direction in each scale 
(see Figure~\ref{fig:shearlet_meanvalue} for an explanation of these kind of figures). 
We clearly see that the imaginary part is non-zero for all non-axis aligned directions. 
The reason is that we destroy the intrinsic symmetry of the Fourier coefficients by cutting out the different parts of the Fourier spectra.
\begin{figure}[htbp]
  \centering
  \includegraphics[width=0.6\textwidth]{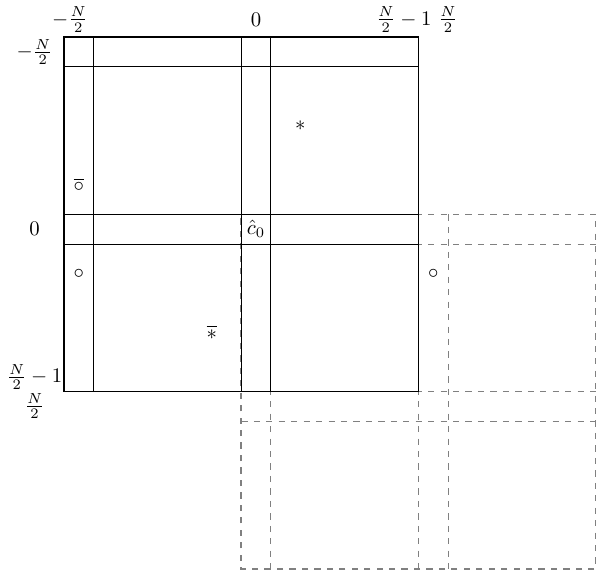}
  \caption{Symmetry of Fourier spectra. The dashed line is the order of the Fourier coefficients provided by \textsc{Matlab}. With \lstinline{fftshift}
  we get the order marked by the solid lines. The point $*$ is an arbitrary point in the interior.
  Due to the symmetry it is complex conjugated to the point $\overline{*}$. The same holds true for the points $\circ$ and $\overline{\circ}$ but they lie on the border of the spectrum.}
  \label{fig:symmetry_Fourier_spectra}
\end{figure}
In Figure~\ref{fig:symmetry_Fourier_spectra} we sketch the Fourier coefficients of a real image of size $N\times N$ where $N$ is even.
The dashed gray square in the lower right part is the output provided by \textsc{Matlab}.
The zero coefficient is stored in the top left corner followed by the positive frequencies up to $\frac{N}{2}-1$ and then the negative frequencies from $\frac{N}{2}$ to $-1$.
By applying \lstinline{fftshift}
the frequencies are swapped such that we obtain the black solid square.
Since the Fourier coefficients are $N$-periodic this can also be seen as shifting the window through the periodic Fourier coefficients.

The Fourier coefficients are symmetric in the following sense: 
Take for example the point symbolized by $*$. 
It has the same real part as the point marked with $\overline{*}$ but with negative imaginary part, i.e., they are complex conjugated. 
This holds true for most of the points.
 
The only points where one has to be careful are those in the first column and the first row since they do not have a corresponding symmetry point. 
As an example we take the point symbolized by $\circ$.
Due to the periodicity of the coefficients $\circ$ appears in the first column but also in the column $\frac{N}{2}$ that does not belong to the image (it only has length $N$). 
But by symmetry the point marked by $\overline{\circ}$ has to be complex conjugated to $\circ$.

By applying the differently oriented (not horizontal or vertical) shearlets we cut out only one of the points $\circ$ or $\overline{\circ}$ such that we loose the symmetry. When we now apply the inverse Fourier transform we get complex shearlet coefficients. 

We propose the following strategy to circumvent this effect:
we modify the shearlets on the finest scale slightly to keep the symmetry. 
Roughly spoken we take the first column (respectively row) of the spectrum of shearlets and mirror it around the central zero axis. 
A complex conjugation is not necessary since we consider only real-valued spectra.
To keep a frame we multiply them by $\frac{1}{\sqrt{2}}$. Figure~\ref{fig:mirroring} illustrates this.
Note that the vertical and horizontal shearlet are kept unchanged.
\begin{figure}[htbp]
  \centering
  \subfloat[Example first row.]{
    \begin{tabular}{|c|c|c|c|c|c|c|c|c|}
    \hline
      8 & 6 & 4 & 2 & 0 & 0 & 0 & 0 & 0 \\
    \hline
    \end{tabular}
  }
  \hspace{0.2cm}
  %
  %
  \subfloat[First row after mirroring and multiplying by $\frac{1}{\sqrt{2}}$.] {
    $\frac{1}{\sqrt{2}}\ \times\ $\begin{tabular}{|c|c|c|c|c|c|c|c|c|}
    \hline
      $8$ & $6$ & $4$ & $2$ & $0$ & $2$ & $4$ & $6$ & $8$ \\
    \hline
    \end{tabular}
    \hspace*{0.2cm}
  }
  \caption{Toy example to illustrate the adjustments to obtain real shearlet spectra.}
  \label{fig:mirroring}
\end{figure}
When we now multiply the complex Fourier coefficients of the given image with the modified spectra, the symmetry is kept.
This leads to real-valued shearlet coefficients or at least to an negligible imaginary part, see Figure~\ref{fig:mean_value_imaginary_part_improved}.

\begin{figure}[htbp]
  \centering
  \subfloat[Mean value of absolute value of the imaginary part of the shearlet coefficients over translations for each scale and shear.]{\label{fig:mean_value_imaginary_part_original}
  \includegraphics[width=0.4\textwidth]{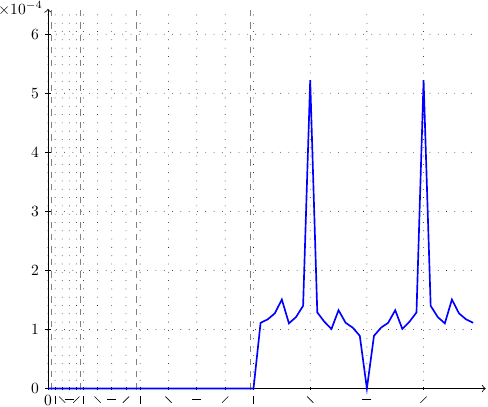}
  \hspace*{0.6cm}}
  \hspace{0.4cm}
  \subfloat[Mean value of absolute value of the imaginary part of the shearlet coefficients over translations for each scale and shear after applying the above described modification.]{\label{fig:mean_value_imaginary_part_improved}
  \includegraphics[width=0.4\textwidth]{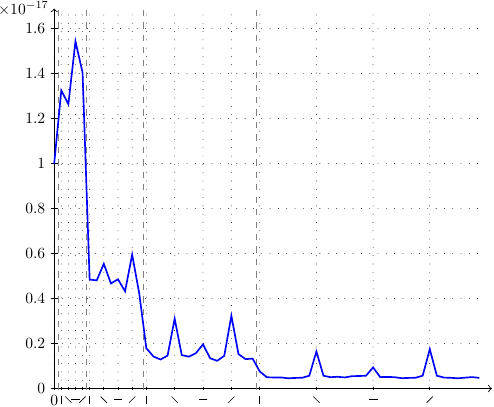}
  \hspace*{0.6cm}}
  \caption{Imaginary part of shearlet coefficients.}
  \label{fig:imaginary_part_shearlet_coefficients}
\end{figure}

\begin{remark}
  Since for odd-sized images the symmetry is always kept, it is also possible to extend the Fourier coefficients by mirroring for even-sized images. 
  But then a FFT of an odd-sized image has to be computed which is in general significantly slower than the one of an even-sized image.
\end{remark}

\subsubsection{Complex Shearlets and Shearlet Coefficients}
In some situations one actually wants \emph{complex shearlets}\index{shearlet!complex} and complex coefficients, in particular for the analysis of the phase of the coefficients (and not only the absolute value).

With our construction complex shearlets (and thus complex shearlet coefficients) can be build straightforward. 
We obtain complex shearlets in time domain by considering one-sided shearlets in frequency domain, see Figure~\ref{fig:complex_shearlet_Fourier}. The resulting shearlets in time domain are shown with their real and imaginary part in Figure~\ref{fig:omplex_shearlets_time}.
The computations and frame properties are the same as in the real-valued case.

\begin{figure}[htbp]
  \centering
  \imageWithBorder{0.4\textwidth}{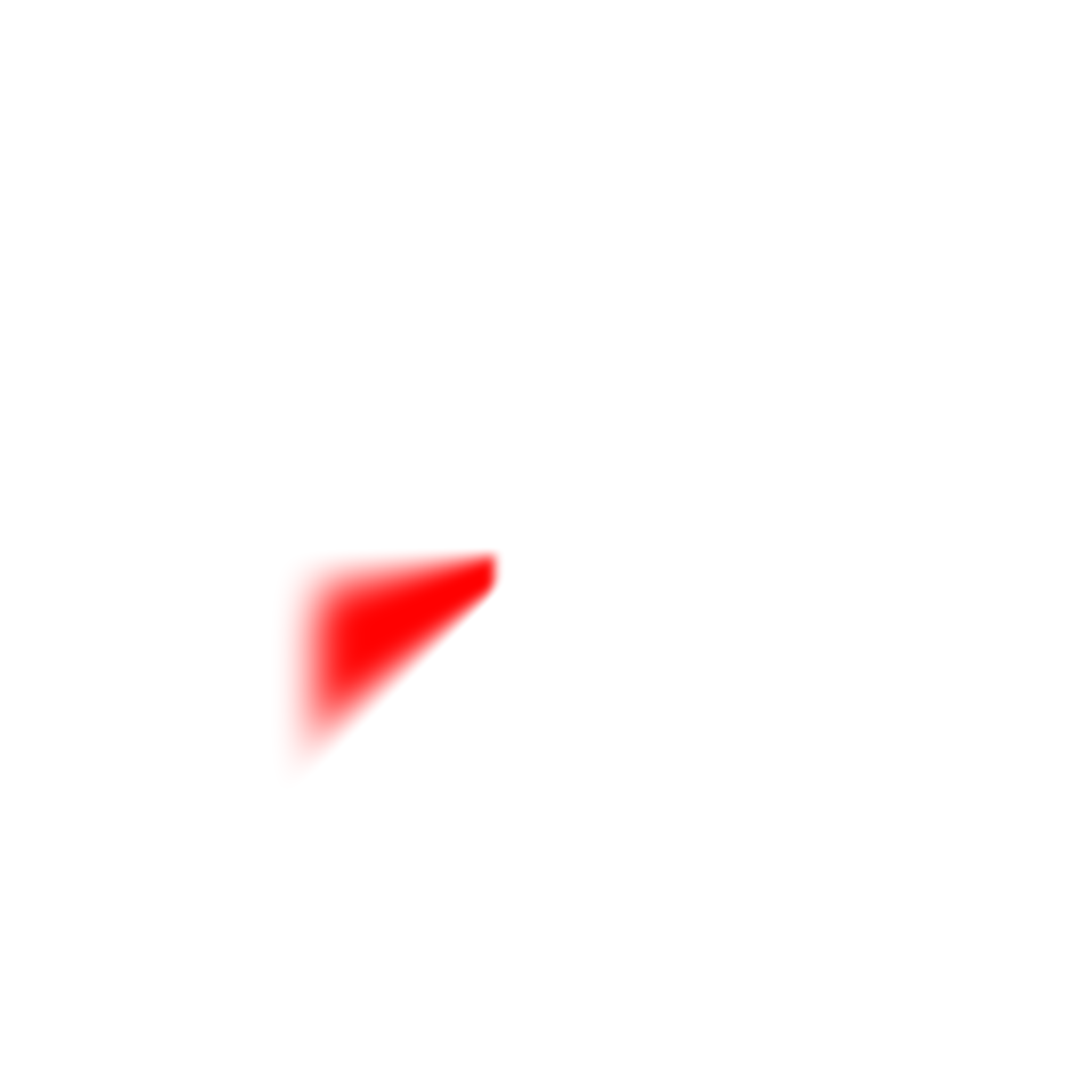}
  \caption{One-sided shearlet in Fourier domain.}
  \label{fig:complex_shearlet_Fourier}
\end{figure}
\begin{figure}[htbp]
  \centering
  \subfloat[Real part of complex shearlet in time domain.]{\label{fig:complex_shearlet_time_realPart}
  \imageWithBorder{0.4\textwidth}{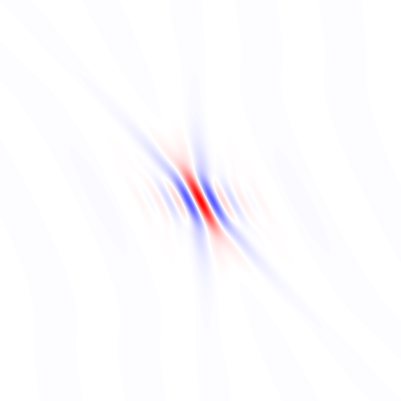}}
  \hspace{0.2cm}
  \subfloat[Imaginary part of complex shearlet in time domain.]{\label{fig:complex_shearlet_time_imagPart}
  \imageWithBorder{0.4\textwidth}{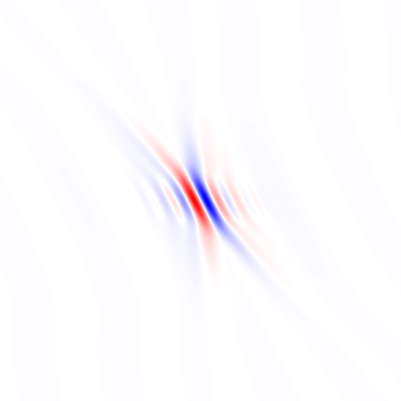}}
\\
  \subfloat[3D-view of real part of complex shearlet in time domain.]{\label{fig:complex_shearlet_time_realPart_3D}
  \imageWithBorder{0.4\textwidth}{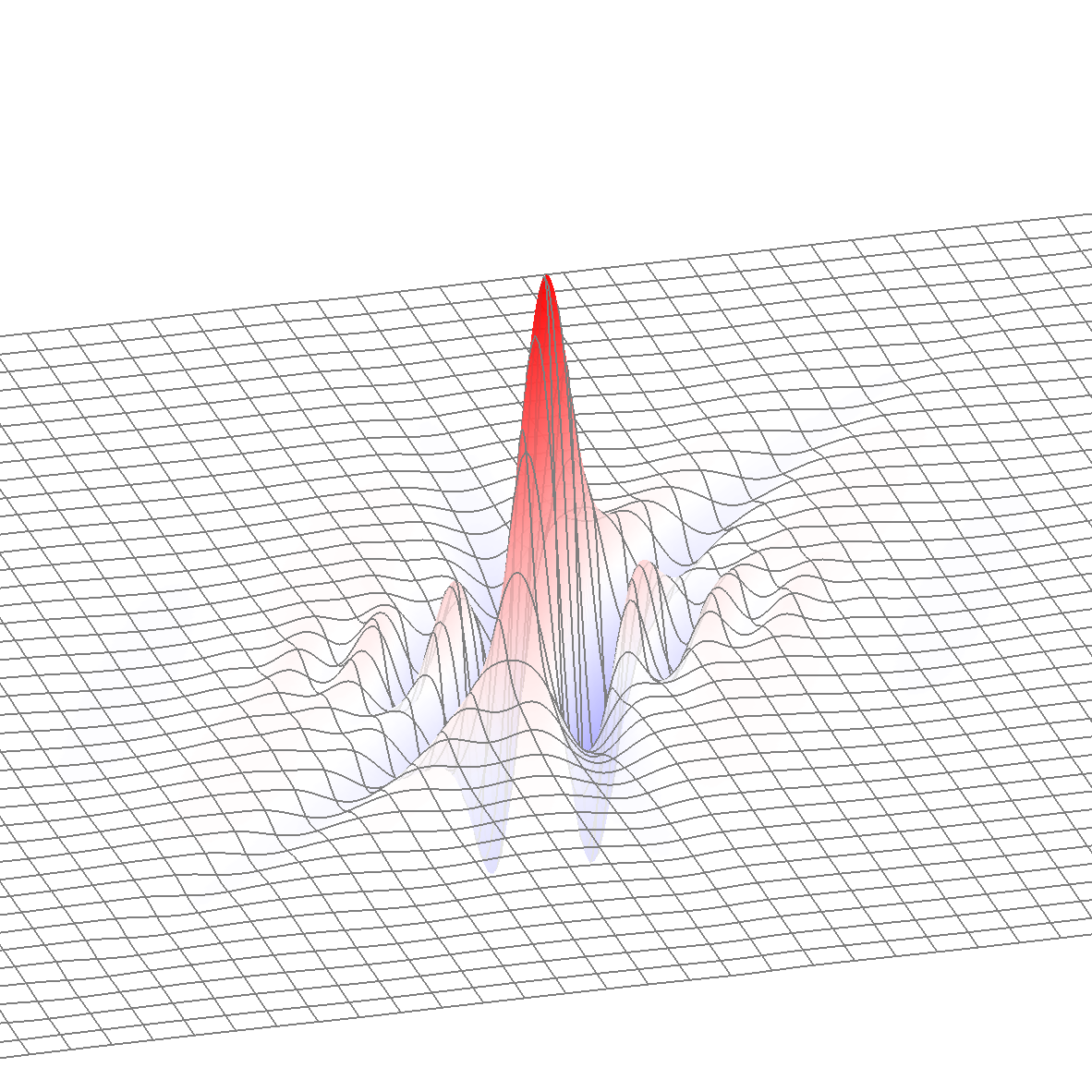}}
  \hspace{0.2cm}
  \subfloat[3D-view of imaginary part of complex shearlet in time domain.]{\label{fig:complex_shearlet_time_imagPart_3D}
  \imageWithBorder{0.4\textwidth}{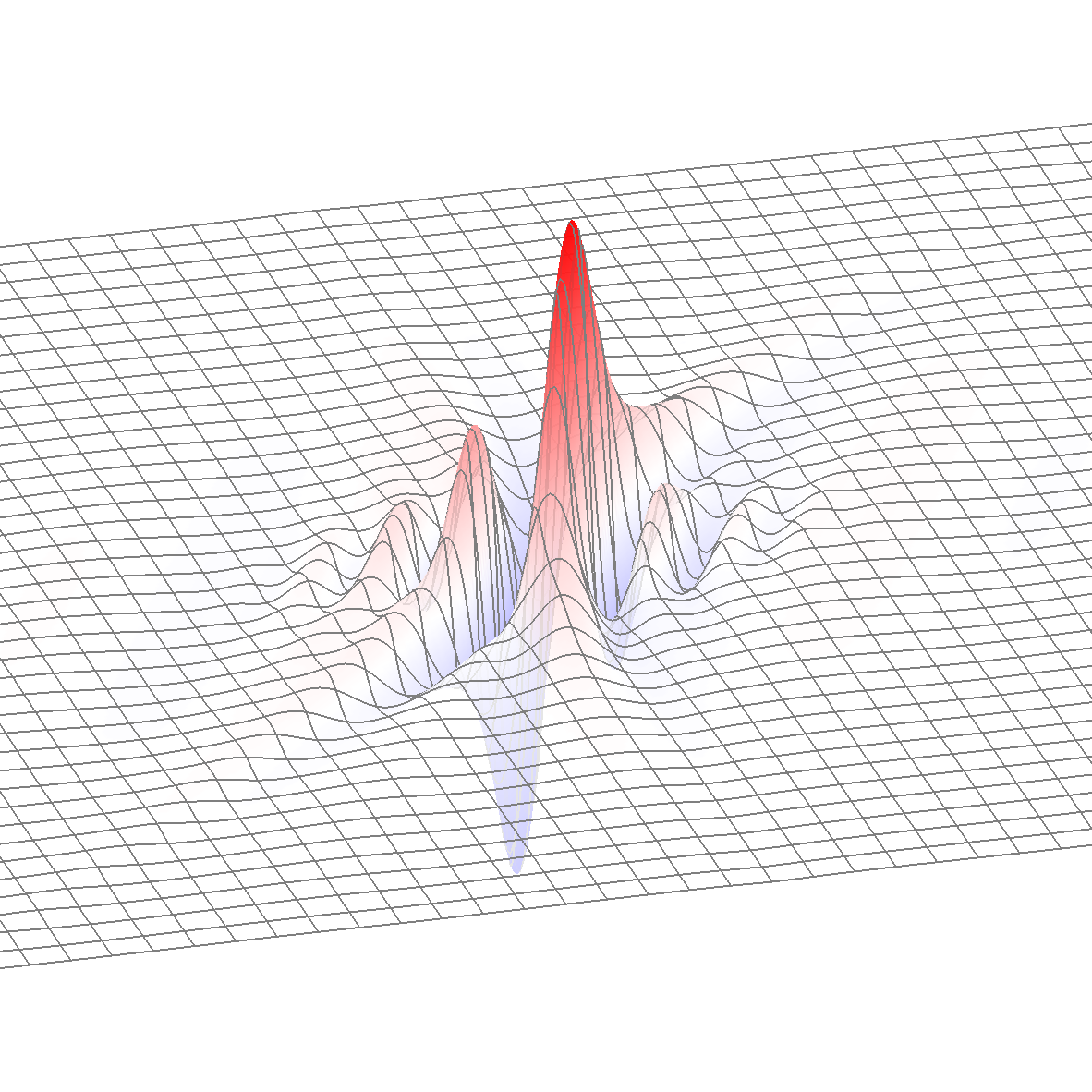}}
  \caption{Complex shearlet in time domain.}
  \label{fig:omplex_shearlets_time}
\end{figure}
\section*{Acknowledgement}
The first author thanks Tomas Sauer (University of Gie\ss{}en) for his support at the beginning of this project.

\bibliographystyle{abbrv}
\bibliography{../diss/latex/bib/extracted}

\end{document}